\documentclass[preprint,12pt]{elsarticle}
\usepackage{amsmath}
\usepackage{amssymb}
\usepackage{lineno}
\usepackage{graphics}
\usepackage[T1]{fontenc}
\usepackage[latin9]{inputenc}
\usepackage{babel}
\usepackage{amsbsy}
\usepackage{amstext}
\usepackage{amsthm}
\usepackage{enumitem}
\usepackage[unicode=true,
 bookmarks=true,bookmarksnumbered=false,bookmarksopen=false,
 breaklinks=false,pdfborder={0 0 1},backref=false,colorlinks=false]
 {hyperref}
\usepackage{breakurl}
\usepackage{mathtools}
\usepackage[mathscr]{eucal}

\makeatletter

  \newtheorem{theorem}{Theorem}[section]
  \newtheorem{corollary}[theorem]{Corollary}
  \newtheorem{lemma}[theorem]{Lemma}
  
  \newtheorem{proposition}[theorem]{Proposition}

  \theoremstyle{definition}
  \newtheorem{definition}[theorem]{Definition}

\def\ps@pprintTitle{%
   \let\@oddhead\@empty
   \let\@evenhead\@empty
   \let\@oddfoot\@empty
   \let\@evenfoot\@oddfoot
}
\begin{document}
\begin{frontmatter}
\title{Large sets near idempotent and its product}

\author[1]{Surajit Biswas\corref{cor1}}
\ead{241surajit@gmail.com}
\author[2]{Sourav Kanti Patra}
\ead{souravkantipatra@gmail.com}
\cortext[cor1]{Corresponding author}
\address[1]{Department of Mathematics, Ramakrishna Mission Vidyamandira, Belur Math, Howrah, West Bengal, India, 711202.}
\address[2]{Department of Mathematics, Centre for Distance And Online Education, The University of Burdwan, Golapbag, Purba Bardhaman, West Bengal, India, 713104.}
\begin{abstract}
Tootkaboni and Vahed introduced the notion of some large sets near idempotent along with some combinatorial properties. We characterize when the finite Cartesian product of central sets near idempotent is central near idempotent. Moreover, we provide a partial characterization for the infinite Cartesian product of the same. We then study the abundance of some large sets near idempotent. We also investigate the effect of tensor product near zero. Finally, as an application we provide a characterization of members of polynomials (with constant term 0) evaluated at idempotents in a near zero semigroup.
\end{abstract}
\end{frontmatter}
\section{Introduction}
The notion of "central subset" was first introduced by Furstenberg \cite{f} in the semigroup $(\mathbb{N}, +)$ of natural numbers. Later, Bergelson and Hindman \cite{bh1} defined the notion of a central set in an arbitrary semigroup. The class of central sets is an important class of combinatorially rich large sets, as they satisfy the Central Sets Theorem \cite[Corollary 14.14.10]{hs}. A localized notion of central set, called central set near idempotent was introduced by Tootkaboni and Vahed in \cite{tv}. In a commutative semigroup, central sets near idempotent is a class of localized sets around an idempotent, those are combinatorially rich, because they satisfy the Central Sets Theorem near idempotent \cite[Theorem 4.3]{tv}.
In \cite{hs2}, Hindman and Strauss have characterized when the (finite and infinite) Cartesian product of sets is piecewise syndetic or central. In \cite{bbp}, similar (partial) characterization has been provided for some other class of large sets in an arbitrary semigroup. In Theorem \ref{2.21} and Theorem \ref{2.22}, we provide similar partial characterizations for piecewise syndetic set near idempotent and central set near idempotent, respectively.

On the other hand, central sets can also be studied from a view point of Ramsey theory. The main goal of Ramsey theory is to study the structures which can not be broken under a finite partition of the system. Van der Waerden's theorem \cite{v} is one such result, which says that, under any finite partition of an infinite arithmetic progression, one of the cells always contains arbitrary long (finite) arithmetic progressions. A deeper observation reveals that for any large set $C$ in the semigroup $(\mathbb{N}, +)$, the largeness of the set $R=\{(a, d)\in\mathbb{N}\times\mathbb{N}:\{a, a+d, \ldots, a+(k-1)d\}\subseteq C\}$ in the semigroup $(\mathbb{N}\times\mathbb{N}, +)$ are of special interest, which leads to the study of abundance of large sets. In \cite{tv}, Tootkaboni and Vahed introduces the notion of some largeness near idempotent. We study the abundance of some large sets near idempotent in Theorem \ref{3.13}. Similar results for near zero semigroup is provided in Corollary \ref{Corollary 3.13} and Theorem \ref{Theorem 3.14}. The matrix version of these results are deduced in \cite{hs3} and \cite{hs4}. In Theorem \ref{3.17} and \ref{3.19}, we provide the near zero analogue of these results. 

Another important result in Ramsey theory is Milliken-Taylor theorem, which ensures the partition regularity of a very general class of infinite configurations called Milliken-Taylor systems. In \cite{bhw}, authors extended these systems to images of very general extended polynomials using the tensor product of ultrafilters. In Section \ref{4}, we begin with the study of tensor product and Lemma \ref{4.5} ensures that the near zero property gets preserved under the tensor product. In Theorem \ref{Theorem 4.9}, we provide a relationship between Milliken-Taylor System and a linear form in one variable evaluated at an ultrafilter in a near zero semigroup. In Theorem \ref{4.11}, we provide a combinatorial characterization of members of polynomials (with constant term 0) evaluated at idempotents for a near zero subsemigroup of both $([0,\infty),+)$ and $([0,\infty),\cdot)$. Finally, in Theorem \ref{4.12} and \ref{4.13}, we conclude with two multidimensional Ramsey theoretic results near zero.

The notion of central set near idempotent is algebraically defined in terms of Stone-\v{C}ech compactification of the corresponding semigroup. For a given semigroup $S$, the Stone-\v{C}ech compactification of $S_d$ (Given a topological space $X$, the notation $X_d$ represents the set $X$ with the discrete topology), i.e. $\beta S_d$ can be naturally identified with the collection of all ultrafilters on $S$ \cite[Theorem 3.27]{hs}, and the semigroup structure on $S$ induces a unique semigroup structure on $\beta S_d$ so that $\beta S_d$ becomes a right topological semigroup with $S$ contained in its topological center \cite[Theorem 4.1]{hs}. This semigroup structure is explicitly given by, for $p, q\in\beta S_d$, $p+q:=\{A\subseteq S:\{s\in S:-s+A\in q\}\in p\}$. From now onwards, whenever we consider Stone-\v{C}ech compactification of a semigroup $S$, we write $\beta S$ for $\beta S_d$ and identify with the ultrafilters on $S$. For other details regarding Stone-\v{C}ech compactification, see \cite{hs}.

\section{Cartesian product near Idempotent}
Let $(T,+)$ be a semitopological semigroup. Throughout this article, we assume all our semitopological semigroups to be Hausdorff. Note that the semitopological structure on the semigroup $T$ is the topology on $T$ such that for each $x\in T$, the left translation map $L_x$ and the right translation map $R_x$ are continuous, where $L_x(y):=x+y$ and $R_x(y):=y+x$, $y\in T$. Let $S$ be a subsemigroup of $T$. We denote the collection of idempotents in $T$
by $E(T)$. Moreover, for every $x\in T$, $\tau_{x}$ denotes the
collection of all neighborhoods of $x$, where a set $U$ is called
a neighborhood of $x\in T$ if $x\in\text{{int}}_{T}(U)$, i.e. if
$x$ is an interior point of $U$. Following \cite[Definition 2.1(a)]{tv}, here we define the collection of ultrafilters on $S$ near an idempotent.
\begin{definition}
Let $S$ be a subsemigroup of a semitopological semigroup $T$. For $e\in E(T)\setminus S$ we denote $e_S^*=\big\{ p\in\beta S:e\in\bigcap_{A\in p}cl_{T}A\big\}$.
\end{definition}
Observe that $e_S^*=\{p\in\beta S: U\cap S\in p \text{ for each } U\in\tau_e\}$.
\begin{definition} 
Let $S$ be a subsemigroup of a semitopological semigroup $T$, and let $e\in E(T)\setminus S$. $S$ is said to be near $e$ subsemigroup of $T$ if $e\in cl_T(S)$.
\end{definition}

For example, $((0,\infty),+)$ is a near $0$ subsemigroup of $([0,\infty),+)$. The semigroup $((0,1),\cdot)$ is both a near 0 subsemigroup and a near $1$ subsemigroup of $([0,1],\cdot)$. Note that here we consider the above semigroups with the usual Euclidean topology. These examples are already introduced and studied in \cite{hl}. For a noncommutative example, consider the Heisenberg group $H:=\big\{[x\,\,y\,\,z]:x,y,z\in\mathbb{R}\big\}$ (where we use the notation $[x\,\,y\,\,z]$ to denote the upper triangular matrix $\begin{pmatrix}
1 &x &z\\
  &1 &y\\
  &  &1
\end{pmatrix}$) with the usual matrix multiplication, and with the topology induced via its identification with $\mathbb{R}^3$ given by $[x\,\,y\,\,z]\longleftrightarrow (x,y,z)$, $(x,y,z)\in\mathbb{R}^3$. Then $H$ is a nilpotent real Lie group. Let $I=[0\,\,0\,\,0]$. Then the subsemigroup $H_{+}:=\big\{[x\,\,y\,\,z]:x,y,z\in (0,\infty)\}$ is a near $I$ subsemigroup of $H$.

\begin{lemma}
If $S$ is a near $e$ subsemigroup of a semitopological semigroup $T$, then $e_S^*$ is a compact subsemigroup of $\beta S$.
\end{lemma}

\begin{proof}
If $e\in cl_T(S)$, pick by \cite[Theorem 3.11]{hs} an ultrafilter $p\in \beta S$ so that $\{U\cap S:U\in\tau_e\}\subseteq p\subseteq \{A\subseteq S:e\in cl_T(A)\}$. Hence $p\in e_S^*$, i.e. $e_S^*$ is nonempty. The fact that $e_S^*$ is a subsemigroup of $\beta S$ follows from \cite[Lemma 2.3(i)]{tv}. 

Finally, observe that $p\in e_S^*$ if and only if $\{U\cap S:U\in\tau_e\}\subseteq p$, i.e. if and only if $p\in \bigcap_{U\in\tau_e}cl_{\beta S}(U\cap S)$. Thus $e_S^*=\bigcap_{U\in\tau_e}cl_{\beta S}(U\cap S)$, and therefore $e_S^*$ is compact. 
\end{proof}

By \cite[Theorem 2.8]{hs}, any compact right topological semigroup $\mathcal{S}$ has a smallest ideal $K(\mathcal{S})$ which is the union of all minimal left ideals of $\mathcal{S}$ and also the union of all minimal right ideals of $\mathcal{S}$. In the next definition, first we recall the notion
of central set near an idempotent \cite[Definition 3.9]{tv}, and then introduce the notion of $central^*$ set near idempotent. 
\begin{definition}\label{2.3}
Let $S$ be a near $e$ subsemigroup of a semitopological semigroup $T$. Let $K(e_S^*)$
be the smallest ideal of the semigroup $e_S^*$. Let $A\subseteq S$.
\begin{enumerate}
 \item The set $A$ is central near $e$ if there is some idempotent $p\in K(e_S^*)$ with $A\in p$.
 \item The set $A$ is $central^*$ near $e$ if $A$ intersects every central set near $e$. 
\end{enumerate}
\end{definition}
For a given sequence $\langle x_n\rangle_{n=1}^\infty$ in a noncommutative semigroup, and for $F\in\mathcal{P}_f(\mathbb{N})$ we define $\sum_{n\in F}x_n$ to be the sum in increasing order of indices, i.e. if $F=\{n_1<n_2<\cdots <n_k\}\in\mathcal{P}_f(\mathbb{N})$, then $\sum_{n\in F}x_n:=x_{n_1}+x_{n_2}+\cdots +x_{n_k}$. Moreover, for $F,G\in\mathcal{P}_f(\mathbb{N})$, we denote $F<G$ if $\max F<\min G$. When we say that a sequence $\langle H_n\rangle_{n=1}^\infty$ in $\mathcal{P}_f(\mathbb{N})$ is an increasing sequence we mean that for each $n\in\mathbb{N}$, $H_n<H_{n+1}$. We recall the following notion of finite sum ("$FS$" in short) \cite[Definition 5.1(b)]{hs} and sum subsystem \cite[Definition 1.3]{bhw}.
\begin{definition}
Let $(S,+)$ be a semigroup. Given a sequence $\langle x_n\rangle_{n=1}^\infty$ in $S$, $FS\big(\langle x_n\rangle_{n=1}^\infty\big)=\big\{\sum_{n\in F}x_n: F\in\mathcal{P}_f(\mathbb{N})\big\}$.
\end{definition}
\begin{definition}
In a semigroup $(S,+)$, $FS\big(\langle y_n\rangle_{n=1}^\infty\big)$ is a sum subsystem of $FS\big(\langle x_n\rangle_{n=1}^\infty\big)$ if there exists an increasing sequence $\langle H_n\rangle_{n=1}^\infty$ in $\mathcal{P}_f(\mathbb{N})$ such that for each $n\in\mathbb{N}$, $y_n=\sum_{t\in H_n}x_t$.
\end{definition}
The following notion of convergence \cite[Definition 3.1]{tv} will be used in the definition of $IP$ set and $IP^*$ set near idempotent.
\begin{definition}
Let $S$ be a near $e$ subsemigroup of a semitopological semigroup $(T,+)$. Let $\langle x_n\rangle_{n=1}^\infty$ be a sequence in $S$. We say $\sum_{n=1}^\infty x_n$ converges near $e$ if for each $U\in\tau_e$ there exists $m\in\mathbb{N}$ such that $FS(\langle x_n\rangle_{n=m}^\infty)\subseteq U$.
\end{definition}

\begin{definition}
Let $S$ be a near $e$ subsemigroup of a semitopological semigroup $(T,+)$. Let $A\subseteq S$.
\begin{enumerate}
    \item The set $A$ is $IP$ near $e$ if there is some sequence $\langle x_n\rangle_{n=1}^{\infty}$ in $S$ such that $\sum_{n=1}^{\infty}x_n$ converges near $e$ and $FS\big(\langle x_n\rangle_{n=1}^{\infty}\big)\subseteq A$.
    \item The set $A$ is $IP^*$ near $e$ if $A$ intersects all the $IP$ sets near $e$.
\end{enumerate}
\end{definition}
For a topological space $X$ with $x\in X$, we say that $x$ has a countable local base in $X$ if $x$ has a countable neighborhood base $\langle U_n\rangle_{n=1}^\infty$ in $X$. If $\langle U_n\rangle_{n=1}^\infty$ is a countable neighborhood base for $x$, one can let $V_n=\bigcap_{k=1}^n U_k$ so that $V_{n+1}\subseteq V_n$ for each $n\in\mathbb{N}$. Therefore, for a countable local base $\langle U_n\rangle_{n=1}^\infty$ we will always assume that $U_{n+1}\subseteq U_n$ for each $n\in\mathbb{N}$. The following theorem provides an algebraic characterization of $IP$ sets near idempotent. (Given a set $Y$ we write $\mathcal{P}_f(Y)$ for the set of finite nonempty subset of $Y$.)
\begin{theorem}\label{Theorem 2.8}
Let $S$ be a near $e$ subsemigroup of a semitopological semigroup $(T,+)$ and let $A\subseteq S$.
\begin{enumerate}
\item If there is a sequence $\langle x_n\rangle_{n=1}^\infty$ such that $\sum_{n=1}^\infty x_n$ converges near $e$ and $FS(\langle x_n\rangle_{n=1}^\infty )\subseteq A$, then there is an idempotent $p\in e_S^*$ such that $A\in p$.
\item If $e$ has a countable local base in $T$ and there is an idempotent $p\in e_S^*$ such that $A\in p$, then there is a sequence $\langle x_n\rangle_{n=1}^\infty$ such that $\sum_{n=1}^\infty x_n$ converges near $e$ and $FS(\langle x_n\rangle_{n=1}^\infty)\subseteq A$.
\end{enumerate}
\end{theorem}
\begin{proof}
 \begin{enumerate}
  \item Let $J=\bigcap_{m=1}^\infty cl_{\beta S}(FS(\langle x_n\rangle_{n=m}^\infty))$. By \cite[Lemma 5.11]{hs}, $J$ is a closed subsemigroup of $\beta S$ so there is an idempotent $p\in J$. It suffices to show that $p\in e_S^*$ so suppose instead that $p\notin e_S^*$. Pick $B\in p$ such that $e\notin cl_T(B)$, and pick $U\in\tau_e$ such that $U\cap B=\emptyset$. Pick $m\in\mathbb{N}$ such that $FS(\langle x_n\rangle_{n=m}^\infty)\subseteq U$. This is a contradiction because $FS(\langle x_n\rangle_{n=m}^\infty)\in p$ and $B\in p$.
  \item Let $\langle W_n\rangle_{n=1}^\infty$ be a local base at $e$. Pick an idempotent $p\in e_S^*$ such that $A\in p$. By \cite[Lemma 4.14]{hs} if $B\in p$, $B^\star=\{x\in B:-x+B\in p\}$, and $x\in B^\star$, then $-x+B^\star\in p$. Note that for each $B\in p$ and each $U\in\tau_e$, $B\cap U\in p$.

Let $B_1=A\cap W_1$. Pick $x_1\in B_1^\star$, let $B_2=(-x_1+B_1^\star)\cap (A\cap W_2)$, and pick $x_2\in B_2^\star$. Inductively let $n\geq 2$ and assume we have chosen $\langle B_t\rangle_{t=1}^n$ and $\langle x_t\rangle_{t=1}^n$ such that if $k\in\{1,2,\ldots,n-1\}$, then $B_{k+1}=(-x_k+B_k^\star)\cap B_k\cap (A\cap W_{k+1})$ and for $k\in\{1,2,\ldots,n\}$, $B_k\in p$ and $x_k\in B_k^\star$. Let $B_{n+1}=(-x_n+B_n^\star)\cap B_n\cap (A\cap W_{n+1})$.

Having chosen $\langle B_t\rangle_{t=1}^\infty$ and $\langle x_t\rangle_{t=1}^\infty$, we claim that if $F\in\mathcal{P}_f(\mathbb{N})$ and $\min F=k$, then $\sum_{t\in F}x_t\in B_k$. We establish this by induction on $|F|$. If $F=\{k\}$, we have $x_k\in B_k$ as required. So assume that $|F|>1$, let $G=F\setminus\{k\}$, and let $m=\min G$. Then $\sum_{t\in G}x_t\in B_m\subseteq B_{k+1}\subseteq (-x_k+B_k)$ so $\sum_{t\in F}x_t\in B_k$.

To see that $\sum_{n=1}^\infty$ converges near $e$, let $U\in\tau_e$. Pick $m$ such that $W_m\subseteq U$. Then $FS(\langle x_n\rangle_{n=m}^\infty)\subseteq B_m\subseteq W_m$.
 \end{enumerate}
\end{proof}

If for each $i\in I$, $X_i$ is a set and $x_i\in X_i$, we will write $\times_{i\in I}x_i$ for the element of $\bigtimes_{i\in I}X_i$ whose $i^{th}$ coordinate is $x_i$.

\begin{lemma}\label{2.7}
Let $I$ be a set. For each $i\in I$, let $S_i$ be a near $e_i$ subsemigroup of a semitopological semigroup $T_i$. Let $T=\bigtimes_{i\in I}T_{i}$, $S=\bigtimes_{i\in I}S_{i}$ and $e=\times_{i\in I}e_{i}$. If for each $i\in I$, $\widetilde{\pi}_i:\beta S\rightarrow \beta S_i$ is the continuous extension of the projection homomorphism $\pi_i:S\rightarrow S_i$, then $\widetilde{\pi}_{i}[e_S^*]=(e_i)_{S_i}^*$. 
\end{lemma}
\begin{proof}
For $p\in e_S^*$, we have $e\in\bigcap_{A\in p}cl_T(A)$.
Also, by \cite[Lemma 3.30]{hs}, $\widetilde{\pi}_{i}(p)=\{A_{i}\subseteq S_{i}:\pi_{i}^{-1}(A_{i})\in p\}$ where $\pi_i^{-1}(A_i)=\{x\in S:\pi_i(x)\in A_i\}\subseteq S$.
Thus, for each $A_{i}\in\widetilde{\pi}_{i}(p)$, we have $\pi_{i}^{-1}(A_{i})\in p$,
which implies $e\in cl_T(\pi_{i}^{-1}(A_{i}))$, i.e. $e_{i}\in cl_{T_{i}}(A_{i})$.
Hence $e_{i}\in\bigcap_{A_{i}\in\widetilde{\pi}_{i}(p)} cl_{T_{i}}(A_{i})$,
i.e. $\widetilde{\pi}_{i}(p)\in (e_i)_{S_i}^*$. Therefore $\widetilde{\pi}_{i}[e_S^*]\subseteq (e_i)_{S_i}^*$. 

To prove the reverse inclusion, let $p_{i}\in (e_i)_{S_i}^*$, i.e. $e_{i}\in\bigcap_{A_{i}\in p_{i}} cl_{T_{i}}(A_{i})$. Consider the family $\mathcal{A}:=\{\pi_i^{-1}(A_i):A_i\in p_i\}$ and the family $\mathcal{R}:=\{A\subseteq S: e\in cl_T(A)\}$. Then $\mathcal{R}$ is nonempty, $\emptyset\notin\mathcal{R}$, $\mathcal{R}^{\uparrow}:=\{B\subseteq S: A\subseteq B \text{ for some } A\in\mathcal{R}\}=\mathcal{R}$ and $\mathcal{R}$ is partition regular. Now, for any nonempty finite collection $\mathcal{F}$ of elements from $p_i$ we have that $\bigcap_{A_i\in\mathcal{F}}A_i\in p_i$, therefore $e_i\in cl_{T_i}\big(\bigcap_{A_i\in\mathcal{F}}A_i\big)$, which implies $e\in cl_T\big(\pi_i^{-1}\big(\bigcap_{A_i\in\mathcal{F}}A_i\big)\big)=cl_T\big(\bigcap_{A_i\in\mathcal{F}}\pi_i^{-1}(A_i)\big)$, and hence $\bigcap_{A_i\in\mathcal{F}}\pi_i^{-1}(A_i)\in\mathcal{R}$. So, by \cite[Theorem 3.11]{hs}, pick $p\in\beta S$ such that $\mathcal{A\subseteq}p\subseteq\mathcal{R}$.
Hence $p\in e_S^*$. Moreover, as $\mathcal{A}\subseteq p$, so for
$A_{i}\in p_{i}$, $\pi_{i}^{-1}(A_{i})\in p$, i.e. $A_{i}\in\widetilde{\pi}_{i}(p)$.
Thus $p_{i}\subseteq\widetilde{\pi}_{i}(p)$, and both being ultrafilters,
we have $p_{i}=\widetilde{\pi}_{i}(p)$. Therefore $\widetilde{\pi}_{i}[e_S^*]=(e_i)_{S_i}^*$.
\end{proof}
If for each $i\in I$, $X_i$ is a topological space, then recall that the box topology on $\bigtimes_{i\in I}X_i$ is the topology generated by the sets $\big\{\bigtimes_{i\in I}U_i:U_i \text{ is open in } X_i\big\}$.
\begin{lemma}\label{2.8}
Let $I$ be a set. For each $i\in I$, let $S_i$ be a near $e_i$ subsemigroup of a semitopological semigroup $T_i$. Let us consider the product semitopological semigroup $T=\bigtimes_{i\in I}T_i$ with box topology, let $S=\bigtimes_{i\in I}S_i$, and $e=\times_{i\in I}e_i$. Consider $\bigtimes_{i\in I}\beta S_i$ with product topology, and let $\tilde{\iota}:\beta S\rightarrow \bigtimes_{i\in I}\beta S_i$ be the continuous extension of the inclusion $\iota:S_d\rightarrow\bigtimes_{i\in I}\beta S_i$. Then $\tilde{\iota}[e_S^*]=\bigtimes_{i\in I}(e_i)_{S_i}^*$.
\end{lemma}
\begin{proof}
For each $i\in I$, if $\pi_i:S\rightarrow S_i$ is the natural projection homomorphism onto the $i^{th}$ coordinate, and if $\widetilde{\pi}_i:\beta S\rightarrow\beta S_i$ is the continuous extension of $\pi_i$, then for each $p\in\beta S$ we have that $\tilde{\pi}_i(p)=\tilde{\iota}(p)_i$, i.e. the $i^{th}$ coordinate of $\tilde{\iota}(p)$ is $\tilde{\pi}_i(p)$. Hence we get $\tilde{\iota}[e_S^*]\subseteq \bigtimes_{i\in I}\tilde{\iota}[e_S^*]_i=\bigtimes_{i\in I}\tilde{\pi}_i[e_S^*]=\bigtimes_{i\in I}(e_i)_{S_i}^*$, where the last eqality follows from Lemma \ref{2.7}.

Conversely, let $\times_{i\in I}p_i\in\bigtimes_{i\in I}(e_i)_{S_i}^*$, i.e. for each $i\in I$ we have $e_i\in\bigcap_{A_i\in p_i} cl_{T_i}(A_i)$. Consider the family $\mathcal{A}:=\{\bigtimes_{i\in I}A_i:A_i\in p_i \text{ for each } i\in I\}$ and the family $\mathcal{R}:=\{A\subseteq S:e\in cl_T(A)\}$. Then $\mathcal{R}$ is nonempty, $\emptyset\notin\mathcal{R}$, $\mathcal{R}^{\uparrow}:=\{B\subseteq S: A\subseteq B \text{ for some } A\in\mathcal{R}\}=\mathcal{R}$ and $\mathcal{R}$ is partition regular. Now, for any nonempty finite collection $\big\langle\bigtimes_{i\in I}A_i^{(n)}\big\rangle_{n=1}^k$ of elements from $\mathcal{A}$, we have that for each $i\in I$, $\bigcap_{n=1}^k A_i^{(n)}\in p_i$, therefore $e_i\in cl_{T_i}\big(\bigcap_{n=1}^k A_i^{(n)}\big)$, and this implies $e\in cl_T\big(\bigtimes_{i\in I}\big(\bigcap_{n=1}^k A_i^{(n)}\big)\big)=cl_T\big(\bigcap_{n=1}^k\big(\bigtimes_{i\in I}A_i^{(n)}\big)\big)$, i.e. $\bigcap_{n=1}^k\big(\bigtimes_{i\in I}A_i^{(n)}\big)\in\mathcal{R}$. Hence, by \cite[Theorem 3.11]{hs}, pick $p\in\beta S$ such that $\mathcal{A}\subseteq p\subseteq\mathcal{R}$. Then $p\in e_S^*$. Moreover, as $\mathcal{A}\subseteq p$, so for each $i\in I$, if $A_i\in p_i$, then $\pi_i^{-1}(A_i)\in\mathcal{A}\subseteq p$, and this implies $A_i\in\widetilde{\pi}_i(p)$. Thus $p_i\subseteq\widetilde{\pi}_i(p)$, and both being ultrafilters, we have $p_i=\widetilde{\pi}_i(p)$. Therefore we have $\tilde{\iota}[e_S^*]=\bigtimes_{i\in I}(e_i)_{S_i}^*$.
\end{proof}
\begin{theorem}\label{2.9}
Let $I$ be a set. For each $i\in I$, let $S_i$ be a near $e_i$ subsemigroup of a semitopological semigroup $T_i$ and let $p_i\in E(K(e_{S_i}^*))$. Let us consider the product semitopological semigroup $T=\bigtimes_{i\in I}T_i$ with box topology. Let $S=\bigtimes_{i\in I}S_i$, $e=\times_{i\in I}e_i$ and $p=\times_{i\in I}p_i$. Consider $\bigtimes_{i\in I}\beta S_i$ with product topology, and let $\tilde{\iota}:\beta S\rightarrow\bigtimes_{i\in I}\beta S_i$ be the continuous extension of the inclusion $\iota: S_d\rightarrow\bigtimes_{i\in I}\beta S_i$. Let $M=\tilde{\iota}^{-1}\{p\}\cap e_S^*$. Then $M$ is a compact subsemigroup of $e_S^*$ and $K(M)\subseteq K(e_S^*)$.
\end{theorem}
\begin{proof}
We have that $\tilde{\iota}$ is a homomorphism by \cite[Corollary 4.22]{hs} and by Lemma \ref{2.8}, $\tilde{\iota}[e_S^*]=\bigtimes_{i\in I}(e_i)_{S_i}^*$. Consequently, $M$ is a compact subsemigroup of $e_S^*$. By \cite[Exercise 1.7.3]{hs}, $\tilde{\iota}[K(e_S^*)]=K(\bigtimes_{i\in I}(e_i)_{S_i}^*)$. Also, by \cite[Theorem 2.23]{hs}, $K(\bigtimes_{i\in I}e_{S_i}^*)=\bigtimes_{i\in I}K((e_i)_{S_i}^*)$ and so $p\in K(\bigtimes_{i\in I}(e_i)_{S_i}^*)$. Consequently, $K(e_S^*)\cap M\neq\emptyset$ and therefore $K(e_S^*)\cap M$ is an ideal of $M$ and so $K(M)\subseteq K(e_S^*)\cap M$.
\end{proof}
From now onwards whenever we say that $A\subseteq S$ is a large set near an idempotent $e$, we always mean that $A$ is a large set in $S$ near $e$, where "large" can be any of central, $central^*$, $IP$, $IP^*$, thick, piecewise syndetic and $PS^*$.
\begin{corollary}\label{2.10}
For each $i\in\{1,2\}$, let $S_i$ be a near $e_i$ subsemigroup of a semitopological
semigroups $T_i$. If for each $i\in\{1,2\}$, $A_i\subseteq S_i$ is central near $e_i$, then $A_1\times A_2\subseteq S_1\times S_2$ is a central set near $(e_{1}, e_{2})$.
\end{corollary}
\begin{proof}
For each $i\in\{1,2\}$, pick an idempotent $p_i\in K((e_i)_{S_i}^*)$ such that $A_i\in p_i$. Let $M$ be as
in Theorem \ref{2.9} and pick an idempotent $r\in K(M)$. Then $\tilde{\iota}(r)=(p_{1}, p_{2})$
and so $A_{1}\times A_{2}\in r$. Since $r\in K((e_1, e_2)_{S_1\times S_2}^*)$,
we have that $A_{1}\times A_{2}\subseteq S_1\times S_2$ is a central set near $(e_{1}, e_{2})$.
\end{proof}

The following notions will be used subsequently. 
\begin{definition}\label{2.11}
Let $S$ be a near $e$ subsemigroup of a semitopological semigroup $(T, +)$. Let $A\subseteq S$. Then $A$ is thick near $e$ if there exists $U\in\tau_{e}$ such that for each $F\in\mathcal{P}_{f}(U\cap S)$ and $V\in\tau_{e}$, there exists $x\in V\cap S$ such that $F+ x\subseteq A$.
\end{definition}

\begin{definition}\label{2.12}(\cite[Definition 3.3(b)]{tv})
Let $S$ be a near $e$ subsemigroup of a semitopological semigroup $(T,+)$. A subset $B$ of $S$ is syndetic near $e$ if for each $U\in\tau_e$, there exists some $F\in\mathcal{P}_f(U\cap S)$ and some $V\in\tau_e$ such that $S\cap V\subseteq \bigcup_{t\in F}(-t+B)$.
\end{definition}

\begin{definition}\label{2.13}
Let $S$ be a near $e$ subsemigroup of a semitopological semigroup $(T,+)$ where $e$ has a countable local base $\langle W_n\rangle_{n=1}^\infty$ in $T$, and let $A\subseteq S$.
 \begin{enumerate}
 \item \label{2.13(1)} The set $A$ is piecewise syndetic near $e$ if there exist sequences $\langle F_n\rangle_{n=1}^\infty$ and $\langle V_n\rangle_{n=1}^\infty$ such that
  \begin{enumerate}
  \item for each $n\in\mathbb{N}$, $F_n\in\mathcal{P}_f(W_n\cap S)$, $V_n\in\tau_e$, and $V_n\subseteq W_n$ and
  \item for each $G\in\mathcal{P}_f(S)$ and for each $O\in\tau_e$, there exists $x\in O\cap S$ such that for each $n\in\mathbb{N}$, $(G\cap V_n)+x\subseteq \bigcup_{t\in F_n}(-t+A)$.
  \end{enumerate}
 \item The set $A$ is $PS^*$ near $e$ if $A$ intersects all the piecewise syndetic sets near $e$.
 \end{enumerate}
\end{definition}

The following proposition shows that the definition of piecewise syndetic set near idempotent is independent of the choice of countable local base.

\begin{proposition}
Let $S$ be a near $e$ subsemigroup of a semitopological semigroup $(T,+)$. Let $\langle W_n^{(1)}\rangle_{n=1}^\infty$ and $\langle W_n^{(2)}\rangle_{n=1}^\infty$ be two countable local bases for $e$ in $T$. If $A\subseteq S$ is piecewise syndetic near $e$ with respect to $\langle W_n^{(1)}\rangle_{n=1}^\infty$, then it is also piecewise syndetic near $e$ with respect to $\langle W_n^{(2)}\rangle_{n=1}^\infty$. 
\end{proposition}

\begin{proof}
Pick a subsequence $\langle W_{n_k}^{(1)}\rangle_{k=1}^\infty$ of $\langle W_n^{(1)}\rangle_{n=1}^\infty$ such that $W_{n_k}^{(1)}\subseteq W_k^{(2)}$ for each $k\in\mathbb{N}$. Pick $\langle F_n^{(1)}\rangle_{n=1}^\infty$ and $\langle V_n^{(1)}\rangle_{n=1}^\infty$ such that
 \begin{enumerate}
 \item for each $n\in\mathbb{N}$, $F_n^{(1)}\in\mathcal{P}_f\big(W_n^{(1)}\cap S\big)$, $V_n^{(1)}\in\tau_e$ and $V_n^{(1)}\subseteq W_n^{(1)}$, and
 \item for each $G\in\mathcal{P}_f(S)$ and for each $O\in\tau_e$, there exists $x\in O\cap S$ such that for each $n\in\mathbb{N}$, $\big(G\cap V_n^{(1)}\big)+x\subseteq \bigcup_{t\in F_n^{(1)}}(-t+A)$.
 \end{enumerate}
 
For each $k\in\mathbb{N}$, let $F_k^{(2)}=F_{n_k}^{(1)}$, and let $V_k^{(2)}=V_{n_k}^{(1)}$. Then it is easy to see that 
 \begin{enumerate}
 \item for each $k\in\mathbb{N}$, $F_k^{(2)}\in\mathcal{P}_f\big(W_k^{(2)}\cap S\big)$, $V_k^{(2)}\in\tau_e$ and $V_k^{(2)}\subseteq W_k^{(2)}$, and
 \item for each $G\in\mathcal{P}_f(S)$ and for each $O\in\tau_e$, there exists $x\in O\cap S$ such that for each $k\in\mathbb{N}$, $\big(G\cap V_k^{(2)}\big)+x\subseteq \bigcup_{t\in F_k^{(2)}}(-t+A)$.
 \end{enumerate}
Hence $A\subseteq S$ is piecewise syndetic near $e$ with respect to $\langle W_n^{(2)}\rangle_{n=1}^\infty$.
\end{proof}

\begin{proposition}
Let $S$ be a near $e$ subsemigroup of a semitopological semigroup $(T,+)$. If $A\subseteq S$ is thick near $e$, then $e\in cl_T(A)$.
\end{proposition}
\begin{proof}
Pick $U\in\tau_e$ as guaranteed for $A$ by the definition. Suppose that $e\notin cl_T(A)$ and pick $W\in\tau_e$ such that $W\cap A=\emptyset$. We may assume that $W$ is open. Now $e+e\in W$ and $T$ is a semitopological so pick $X\in\tau_e$ such that $X+e\subseteq W$. Since $S$ is a near $e$ subsemigroup of $T$, pick $a\in S\cap U\cap X$. Pick $V\in\tau_e$ such that $a+V\subseteq W$. Pick $x\in V\cap S$ such that $a+x\in A$. Then $a+x\in W\cap A$, a contradiction.
\end{proof}

The following proposition provides a natural class of thick sets near idempotent.

\begin{proposition}
Let $S$ be a near $e$ subsemigroup of a semitopological semigroup $(T,+)$. If $U\in\tau_e$, then $U\cap S\subseteq S$ is a thick set near $e$. 
\end{proposition}
\begin{proof}
Without loss of generality we assume that $U$ is open. Now $e+e\in U$ and $T$ is semitopological so pick $W\in\tau_e$ such that $W+e\subseteq U$. For given $F\in\mathcal{P}_f(W\cap S)$ and $V\in\tau_e$, we pick $V_w\in\tau_e$ for each $w\in F\cap W$ such that $w+V_w\subseteq U$. For any chosen $x\in \big(\bigcap_{w\in F\cap W}V_w\big)\cap S\cap V$ we then have $F+x\subseteq U$, i.e. $F+x\subseteq U\cap S$. Hence $U\cap S\subseteq S$ is thick near $e$.
\end{proof}

The next two lemmas provide algebraic analogues of thick set near idempotent and piecewise syndetic set near idempotent.
\begin{lemma}\label{2.17}
Let $S$ be a near $e$ subsemigroup of a semitopological semigroup $(T, +)$. If $A\subseteq S$ is thick near $e$, then there exists a left ideal $L$ of $e_S^*$ such that $L\subseteq cl_{\beta S}A$.
\end{lemma}
\begin{proof}
Pick $U\in\tau_e$ such that for every $F\in\mathcal{P}_f(U\cap S)$ and every $V\in\tau_e$, there exists $x\in V\cap S$ such that $F+x\subseteq A$. Consider the family $\mathcal{A}:=\{-s+A:s\in U\cap S\}$ and the family $\mathcal{R}:=\{B\subseteq S:e\in cl_{T}(B)\}$. Then $\mathcal{R}$ is nonempty, $\emptyset\notin\mathcal{R}$, $\mathcal{R}^{\uparrow}:=\{C\subseteq S:B\subseteq C\text{ for some }B\in\mathcal{R}\}=\mathcal{R}$ and $\mathcal{R}$ is partition regular. We claim that for every $F\in\mathcal{P}_f(U\cap S)$, $\bigcap_{s\in F}(-s+A)\in\mathcal{R}^{\uparrow}$. For $V\in\tau_e$ we pick $x\in V\cap S$ such that $F+x\subseteq A$, so we get $x\in\bigcap_{s\in F}(-s+A)$, and therefore $\big(\bigcap_{s\in F}(-s+A)\big)\cap V\neq\emptyset$, i.e. $e\in cl_T\big(\bigcap_{s\in F}(-s+A)\big)$. This completes the proof of the above claim, and hence pick by \cite[Theorem 3.11]{hs} an ultrafilter $p\in\beta S$ such that $\mathcal{A}\subseteq p\subseteq\mathcal{R}$. Then $p\in e_S^*$. Now we will show that the left ideal $e_S^* +p$ of $e_S^*$ is contained in $cl_{\beta S}A$, i.e. for each $q\in e_S^*$, $A\in q+p$. We know that $A\in q+p$ if and only if $\{s\in S:-s+A\in p\}\in q$.
Now, as $\mathcal{A}\subseteq p$, so $U\cap S\subseteq\{s\in S:-s+A\in p\}$.
But $q\in e_S^*$ implies $U\cap S\in q$, and we have $\{s\in S:-s+A\in p\}\in q$, i.e. $A\in q+p$.
\end{proof}

Note that from the above lemma it is immediate that any thick set near idempotent is a central set near idempotent. The next lemma is generalization of \cite[Theorem 3.5]{hl}.

\begin{lemma}\label{2.18}
Let $S$ be a near $e$ subsemigroup of a semitopological semigroup $(T,+)$ where $e$ has a countable local base $\langle W_n\rangle_{n=1}^\infty$ in $T$, and let $A\subseteq S$. Then $A\subseteq S$ is piecewise syndetic near $e$ if and only if $K(e_S^*)\cap cl_{\beta S}A\neq\emptyset$.
\end{lemma}

\begin{proof}
Necessity. Pick $\langle F_n\rangle_{n=1}^\infty$ and $\langle V_n\rangle_{n=1}^\infty$ satisfying conditions (a) and (b) of Definition \ref{2.13}(\ref{2.13(1)}). Given $G\in\mathcal{P}_f(S)$ and $O\in\tau_e$, let $$C(G,O)=\{x\in O\cap S: \text{for all } n\in\mathbb{N},\,(G\cap V_n)+x\subseteq\bigcup_{t\in F_n}(-t+A)\}.$$
By assumption, each $C(G,O)\neq\emptyset$. Further, given $G_1$ and $G_2$ in $\mathcal{P}_f(S)$ and $O_1,O_2\in\tau_e$ we have that $C(G_1\cup G_2,O_1\cap O_2)\subseteq C(G_1,O_1)\cap C(G_2,O_2)$ so $\{C(G,O):G\in\mathcal{P}_f(S) \text{ and } O\in\tau_e\}$ has the finite intersection property so pick $p\in\beta S$ with $\{C(G,O):G\in\mathcal{P}_f(S) \text{ and } O\in\tau_e\}\subseteq p$. Note that since each $C(G,O)\subseteq O\cap S$, so $p\in e_S^*$.

Now we claim that for each $n\in\mathbb{N}$, $e_S^*+p\subseteq cl_{\beta S}\big(\bigcup_{t\in F_n}(-t+A)\big)$, so let $n\in\mathbb{N}$, and let $q\in e_S^*$. To show that $\bigcup_{t\in F_n}(-t+A)\in q+p$, we show that $$V_n\cap S\subseteq \big\{y\in S:-y+\bigcup_{t\in F_n}(-t+A)\in p\big\}.$$
So let $y\in V_n\cap S$. Then $C(\{y\},V_n)\in p$ and $C(\{y\},V_n)\subseteq -y+\bigcup_{t\in F_n}(-t+A)$.

Now pick $r\in (e_S^*+p)\cap K(e_S^*)$ (since $e_S^*+p$ is a left ideal of $e_S^*$). Given $n\in\mathbb{N}$, $\bigcup_{t\in F_n}(-t+A)\in r$ so pick $t_n\in F_n$ such that $-t_n+A\in r$. Now each $t_n\in F_n\subseteq W_n$ so $\lim_{n\to\infty}t_n=e$ so pick $q\in e_S^*\cap cl_{\beta S}(\{t_n:n\in\mathbb{N}\})$. Then $q+r\in K(e_S^*)$ and $\{t_n:n\in\mathbb{N}\}\subseteq \{t\in S:-t+A\in r\}$ so $A\in q+r$.

Sufficiency. Pick $p\in K(e_S^*)\cap cl_{\beta S}A$. Let $B=\{x\in S:-x+A\in p\}$. By \cite[Theorem 3.4]{tv}, $B$ is syndetic near $e$. Inductively for each $n\in\mathbb{N}$ pick $F_n\in\mathcal{P}_f(W_n\cap S)$ and $V_n\in\tau_e$ (with $V_n\subseteq W_n$ and $V_{n+1}\subseteq V_n$) such that $S\cap V_n\subseteq \bigcup_{t\in F_n}(-t+B)$.

Let $G\in\mathcal{P}_f(S)$ be given. If $G\cap V_1=\emptyset$, the conclusion is trivial, so assume $G\cap V_1\neq\emptyset$ and let $H=G\cap V_1$. For each $y\in H$, let $m(y)=\max\{n\in\mathbb{N}:y\in V_n\}$. For each $y\in H$ and each $n\in\{1,2,\ldots,m(y)\}$, we have $y\in\bigcup_{t\in F}(-t+B)$ so pick $t(y,n)\in F_n$ such that $y\in -t(y,n)+B$. Then given $y\in H$ and $n\in\{1,2,\ldots,m(y)\}$, we have that $-(t(y,n)+y)+A\in p$.

Now let $O\in\tau_e$ be given. Then $O\in p$ so pick $$x\in O\cap\bigcap_{y\in H}\bigcap_{n=1}^{m(y)}(-(t(y,n)+y)+A).$$
Then given $n\in\mathbb{N}$ and $y\in G\cap V_n$, we have that $y\in H$ and $n\leq m(y)$, so $t(y,n)+y+x\in A$, so $y+x\in -t(y,n)+A\subseteq \bigcup_{t\in F_n}(-t+A)$.
\end{proof}

If $S$ is a near $e$ subsemigroup of a semitopological semigroup $T$ where $e$ has a countable local base in $T$, then the above lemma provides that any central set in $S$ near $e$ is a piecewise syndetic set in $S$ near $e$. Moreover, Theorem \ref{Theorem 2.8} provides that any central set in $S$ near $e$ is also an $IP$ set in $S$ near $e$.

\begin{lemma}\label{2.19}
Let $I$ be a set. For $i\in I$, let $S_i$ be a near $e_i$ subsemigroup of a semitopological semigroup $(T_i, +)$ and let $A_i\subseteq S_i$. Let us consider product semitopological semigroup $T=\bigtimes_{i\in I}T_i$ with box topology. Let $S=\bigtimes_{i\in I}S_i, \,A=\bigtimes_{i\in I}A_i$, and $e=\times_{i\in I}e_i$.
\begin{enumerate}
 \item If for each $i\in I$, $A_i\subseteq S_i$ is thick near $e_i$, then $A\subseteq S$ is thick near $e$.
 \item If $A\subseteq S$ is central near $e$, then $A_i\subseteq S_i$ is central near $e_i$ for each $i\in I$.
 \item If $e$ has a countable local base in $T$ and if $A\subseteq S$ is piecewise syndetic near $e$, then $A_i\subseteq S_i$ is piecewise syndetic near $e_i$ for each $i\in I$.
\end{enumerate}
\end{lemma}
\begin{proof}
For each $i\in I$, let $\pi_i:S\rightarrow S_i$ be the projection homomorphism from $S$ onto $S_i$.
\begin{enumerate}
 \item For each $i\in I$, pick $U_i\in\tau_{e_i}$ such that for each $F_i\in\mathcal{P}_f(U_i\cap S_i)$ and $V_i\in\tau_{e_i}$, there exists $x_i\in V_i\cap S_i$ such that $F_i+x_i\subseteq A_i$. Let $U=\bigtimes_{i\in I}U_i$. Let $F\in\mathcal{P}_f(U\cap S)$ and $V\in\tau_e$, and pick $V_i\in\tau_{e_i}$ for each $i\in I$ such that $\bigtimes_{i\in I}V_i\subseteq V$. 
 
Now, for $\pi_i[F]\in\mathcal{P}_f(S_i)$ and $V_i\in\tau_{e_i}$ pick $x_i\in V_i\cap S_i$ such that $\pi_i[F]+x_i\subseteq A_i$. Let $x=\times_{i\in I}x_i$. Then $$F+x\subseteq \bigtimes_{i\in I}(\pi_i[F]+x_i)\subseteq \bigtimes_{i\in I}A_i=A,$$ 
and therefore $A\subseteq S$ is thick near $e$.
 \item For each $i\in I$, let $\widetilde{\pi}_{i}:\beta S\rightarrow\beta S_{i}$ be the continuous extension of $\pi_{i}$. By Lemma \ref{2.7}, each $\widetilde{\pi}_{i}|_{e_S^*}:e_S^*\rightarrow (e_i)_{S_i}^*$ is a surjective homomorphism. Hence, by \cite[Exercise 1.7.3]{hs} we have $\widetilde{\pi}_{i}[K(e_S^*)]=K((e_i)_{S_i}^*)$. Pick an idempotent $p\in K(e_S^*)$ such that $A\in p$. Then $\widetilde{\pi}_{i}(p)$ is an idempotent in $K((e_i)_{S_i}^*)$ and $A_{i}\in \widetilde{\pi}_i(p)$.
 \item Let $\langle W_n\rangle_{n=1}^\infty$ be a countable local base for $e$ in $T$. For each $n\in\mathbb{N}$ and for each $i\in I$ pick a neighborhood $W_i^{(n)}$ of $e_i$ in $T_i$ such that $e\in\bigtimes_{i\in I}W_i^{(n)}\subseteq W_n$. Then $\big\langle\bigtimes_{i\in I}W_i^{(n)}\big\rangle_{n=1}^\infty$ is a countable local base for $e$ in $T$ and for each $i\in I$, $\big\langle W_i^{(n)}\big\rangle_{n=1}^\infty$ is a countable local base for $e_i$ in $T_i$. With respect to the countable local base $\big\langle\bigtimes_{i\in I}W_i^{(n)}\big\rangle_{n=1}^\infty$ for $e$ in $T$, pick $\langle F_n\rangle_{n=1}^\infty$ and $\langle V_n\rangle_{n=1}^\infty$ as guaranteed for $A$ by Definition \ref{2.13}(\ref{2.13(1)}). But, then the conditions $(a)$ and $(b)$ of Definition \ref{2.13}(\ref{2.13(1)}) also hold if we replace $V_n$ by any element of $\tau_e$ contained in $V_n$ for all $n\in\mathbb{N}$. So without loss of generality we can assume that for all $n\in\mathbb{N}$, $V_n=\bigtimes_{i\in I} V_i^{(n)}$. 

Fix $i\in I$. Then for each $n\in\mathbb{N}$ we have that $\pi_i[F_n]\in\mathcal{P}_f\big(W_i^{(n)}\cap S_i\big)$, $V_i^{(n)}\in\tau_{e_i}$ and $V_i^{(n)}\subseteq W_i^{(n)}$. Now let $G_i\in\mathcal{P}_f(S_i)$ and let $O_i\in\tau_{e_i}$. Then fix some element $s_j\in S_j$ for each $j\neq i$ and define $G=\bigtimes_{j\in I}G_j$, where $G_j:=\{s_j\}$ for $j\neq i$. Also, define $O=\bigtimes_{j\in I}O_j$, where $O_j:=T_j$ for each $j\neq i$. Pick $x=\times_{j\in I}x_j\in O\cap S$ such that for each $n\in\mathbb{N}$, $(G\cap V_n)+x\subseteq\bigcup_{t\in F_n}(-t+A)$, which implies $$\bigtimes_{j\in I}\big(\big(G_j\cap V_j^{(n)}\big)+x_j\big)\subseteq \bigcup_{t\in F_n}(-t+A)\subseteq\bigtimes_{j\in I}\bigcup_{t\in\pi_j[F_n]}(-t+A_j),$$
and in particular we have for each $n\in\mathbb{N}$ that $\big(G_i\cap V_i^{(n)}\big)+x_i\subseteq\bigcup_{t\in\pi_i[F_n]}(-t+A_i)$ where $x_i\in O_i\cap S_i$. Hence $A_i\subseteq S_i$ is piecewise syndetic near $e_i$. 
\end{enumerate}
\end{proof}
The proof of the following lemma directly follows from Definition \ref{2.13}(\ref{2.13(1)}).
\begin{lemma}\label{2.20}
For each $i\in\{1,2\}$, let $S_i$ be a near $e_i$ subsemigroup of a semitopological semigroup $T_i$ where $e_i$ has a countable local base in $T_i$. If for each $i\in\{1,2\}$, $A_i\subseteq S_i$ is piecewise syndetic near $e_i$, then $A_1\times A_2\subseteq S_1\times S_2$ is piecewise syndetic near $(e_1, e_2)$.
\end{lemma}
Now we have the following partial characterization of when the Cartesian product of piecewise syndetic sets near idempotent is a piecewise syndetic set near idempotent.
\begin{theorem}\label{2.21}
Let $I$ be a set. For each $i\in I$, let $S_i$ be a near $e_i$ subsemigroup of a semitopological semigroup $T_i$ and let $A_i\subseteq S_i$. Let us consider the product semitopological semigroup $T=\bigtimes_{i\in I}T_i$ with box topology. Let $S=\bigtimes_{i\in I}S_i$, $A=\bigtimes_{i\in I}A_i$ and $e=\times_{i\in I}e_i$. Let $e$ have a countable local base in $T$. Let $J=\{i\in I:A_i\subseteq S_i  \text{ is not thick near } e_i\}$ be a finite set. Then $A\subseteq S$ is piecewise syndetic near $e$ if and only if for each $i\in I$, $A_i\subseteq S_i$ is piecewise syndetic near $e_i$.  
\end{theorem}
\begin{proof}
The forward implication follows from Lemma \ref{2.19}. For the converse part, if $J=I$, then delete the references to $I\setminus J$ in the part that follows. Similarly, if $J=\emptyset$, then delete the references to $J$ in the part that follows. By Lemma \ref{2.19}, $\bigtimes_{i\in I\setminus J}A_i\subseteq \bigtimes_{i\in I\setminus J}S_i$ is thick near $\times_{i\in I\setminus J}e_i$, and therefore it is piecewise syndetic near $\times_{i\in I\setminus J}e_i$. Moreover, by inductive application of Lemma \ref{2.20}, $\bigtimes_{i\in J}A_i\subseteq \bigtimes_{i\in J}S_i$ is piecewise syndetic near $\times_{i\in J}e_i$. Thus $\big(\bigtimes_{i\in I\setminus J}A_i\big)\times\big(\bigtimes_{i\in J}A_i\big)\subseteq \big(\bigtimes_{i\in I\setminus J}S_i\big)\times\big(\bigtimes_{i\in J}S_i\big)$ is piecewise syndetic near $\big(\times_{i\in I\setminus J}e_i\big)\times\big(\times_{i\in J}e_i\big)$, i.e. $A\subseteq S$ is piecewise syndetic near $e$.
\end{proof}
Regarding the Cartesian product of Central sets near idempotent, we have the following theorem.
\begin{theorem}\label{2.22}
Let $I$ be a set. For each $i\in I$, let $S_i$ be a near $e_i$ subsemigroup
of a semitopological semigroup $T_i$ and let $A_{i}\subseteq S_{i}$. Let us consider the product semitopological semigroup $T=\bigtimes_{i\in I}T_i$ with box topology. Let $S=\bigtimes_{i\in I}S_{i}$,
$A=\bigtimes_{i\in I}A_{i}$ and $e=\times_{i\in I}e_{i}$. Let
$J=\{i\in I:A_i\subseteq S_i \text{ is not thick near } e_i\}$ be a finite set.
Then $A\subseteq S$ is central near $e$ if and only if for each $i\in I$,
$A_i\subseteq S_i$ is central near $e_i$. 
\end{theorem}
\begin{proof}
The forward implication follows from Lemma \ref{2.19}. For the converse part, if $J=I$, then delete the references to $I\setminus J$ in the part that follows. Similarly, if $J=\emptyset$, then delete the references to $J$ in the part that follows. By
Lemma \ref{2.19}, $\bigtimes_{i\in I\setminus J}A_{i}\subseteq \bigtimes_{i\in I\setminus J}S_i$ is thick near
$\times_{i\in I\setminus J}e_{i}$, and therefore it is central near $\times_{i\in I\setminus J}e_{i}$. Moreover, by inductive application of Corollary \ref{2.10}, $\bigtimes_{i\in J}A_i\subseteq \bigtimes_{i\in J}S_i$ is central near $\times_{i\in J}e_i$. Thus $\big(\bigtimes_{i\in I\setminus J}A_{i}\big)\times\big(\bigtimes_{i\in J}A_{i}\big)\subseteq \big(\bigtimes_{i\in I\setminus J}S_{i}\big)\times\big(\bigtimes_{i\in J}S_{i}\big)$
is central near $\big(\times_{i\in I\setminus J}e_{i}\big)\times\big(\times_{i\in J}e_{i}\big)$,
i.e. $A\subseteq S$ is central near $e$.
\end{proof}
\section{Abundance of Large sets near idempotent}
In \cite{bh}, the authors have studied the abundance of large sets. Following them, we deduce the corresponding results near idempotent. In the following, we write $I^\Diamond$ for a subsemigroup of $S^l$. When we say that $\varepsilon$ is a property which may be possessed by subsets of a semigroup, we mean properties such as those we have been considering, whose definition depends on the particular semigroup in which the sets reside. By $\varepsilon^*$ set in a semigroup, we mean a subset in the semigroup which intersects all the $\varepsilon$ sets in that semigroup. By $\pi_i$ we mean the projection onto the $i^{th}$ coordinate.

\begin{lemma}\label{3.3}
Let $\varepsilon$ be a partition regular property which may be possessed by subsets of a semigroup. Let $S$ be a semigroup, let $l\in \mathbb{N}$, and let $I^\Diamond$ be a subsemigroup of $S^l$. Statement $(1)$ implies statement $(2)$. If each superset of an $\varepsilon$ set in $S$ is an $\varepsilon$ set, then statements $(1)$ and $(2)$ are equivalent.
\begin{enumerate}
\item For every $\varepsilon$ set $A$ in $I^\Diamond$ and every $i\in\{1, 2, \ldots, l\}$, $\pi_i[A]$ is an $\varepsilon$ set in $S$.
\item Whenever $B$ is an $\varepsilon^*$ set in $S$, $B^l\cap I^\Diamond$ is an $\varepsilon^*$ set in $I^\Diamond$.
\end{enumerate}
\end{lemma}
\begin{proof}
\cite[Theorem 2.2]{bh}.
\end{proof}
 
Note that Definition \ref{2.3}, Theorem \ref{Theorem 2.8} and Lemma \ref{2.18}  ensures the partition regular property of central sets near idempotent, IP sets near idempotent  and piecewise syndetic sets near idempotent, respectively. If $X$ is a set, $x\in X$ and $l\in\mathbb{N}$, we will write $\overline{x}$ for the element in $X^l$ each of whose coordinates is $x$.   
\begin{corollary}\label{3.4}
Let $S$ be a dense  near $e$ subsemigroup of a semitopological semigroup $T$. Let $l\in \mathbb{N}$ and let $I^\Diamond\subseteq S^l$ be a near $\overline{e}$ semigroup of $T^l$. If $B\subseteq S$ be an $IP^*$ set near $e$, then $B^l\cap I^\Diamond\subseteq I^\Diamond$ is an $IP^*$ set near $\overline{e}$.
\end{corollary}
\begin{proof}
It is immediate from the definition that whenever $A\subseteq I^\Diamond$ is an $IP$ set near $\overline{e}$ and $i\in \{1, 2, \ldots, l\}$, then $\pi_i[A]\subseteq S$ is an IP set near $e$. Now the result follows from Lemma \ref{3.3}.
\end{proof}

\begin{lemma}\label{Lemma 3.4}
Let $S$ be a near $e$ subsemigroup of a semitopological semigroup $T$. Let $l\in\mathbb{N}$. Let $I^\Diamond\subseteq S^l$ be a near $\overline{e}$ subsemigroup of $T^l$ such that for each $i\in\{1,2,\ldots, l\}$ and for $\mathcal{E}_i\subseteq\pi_i[I^\Diamond]$, if $e\in cl_T(\mathcal{E}_i)$ then $\overline{e}\in cl_{T^l}\big(\pi_i^{-1}(\mathcal{E}_i)\big)$ where $\pi_i:I^\Diamond\rightarrow\pi_i[I^\Diamond]$ is the natural projection homomorphism onto the $i^{th}$ coordinate. If for each $i\in\{1,2,\ldots,l\}$, $\widetilde{\pi}_i:\beta I^\Diamond\rightarrow\beta(\pi_i[I^\Diamond])$ is the continuous extension of $\pi_i$, then $\widetilde{\pi}_i\big[\overline{e}_{I^\Diamond}^*\big]=e_{\pi_i[I^\Diamond]}^*$.

\end{lemma}
\begin{proof}
Let $i\in\{1,2,\ldots,l\}$. For $p\in\overline{e}_{I^\Diamond}^*$, we have $\overline{e}\in\bigcap_{A\in p}cl_{T^l}(A)$. Also, by \cite[Lemma 3.30]{hs}, $\widetilde{\pi}_i(p)=\{A_i\subseteq\pi_i[I^\Diamond]:\pi_i^{-1}(A_i)\in p\}$. Thus, for each $A_i\in\widetilde{\pi}_i(p)$ we have $\pi_i^{-1}(A_i)\in p$, which implies $\overline{e}\in cl_{T^l}\big(\pi_i^{-1}(A_i)\big)$, i.e. $e\in cl_T(A_i)$. Hence $e\in\bigcap_{A_i\in\widetilde{\pi}_i(p)}cl_T(A_i)$, i.e. $\widetilde{\pi}_i(p)\in e_{\pi_i[I^\Diamond]}^*$.

To prove the reverse inclusion, let $p_i\in e_{\pi_i[I^\Diamond]}^*$. Consider the family $\mathcal{A}:=\{\pi_i^{-1}(A_i):A_i\in p_i\}$ and the family $\mathcal{R}:=\{A\subseteq I^\Diamond:\overline{e}\in cl_{T^l}(A)\}$. Then $\mathcal{A}$ is closed under finite intersection, $\mathcal{R}$ is nonempty, $\emptyset\notin\mathcal{R}$, $\mathcal{R}^{\uparrow}:=\{B\subseteq I^\Diamond:A\subseteq B \text{ for some } A\in\mathcal{R}\}=\mathcal{R}$ and $\mathcal{R}$ is partition regular. Moreover, if $A_i\in p_i$, then $A_i\subseteq\pi_i[I^\Diamond]$ and $e\in cl_T(A_i)$ implies $\overline{e}\in cl_{T^l}\big(\pi_i^{-1}(A_i)\big)$. Thus $\mathcal{A}\subseteq\mathcal{R}$. Hence, by \cite[Theorem 3.11]{hs} pick $p\in\beta I^\Diamond$ such that $\mathcal{A}\subseteq p\subseteq\mathcal{R}$. Then $p\in \overline{e}_{I^\Diamond}^*$. Also, if $A_i\in p_i$, then $\pi_i^{-1}(A_i)\in p$, i.e. $A_i\in\widetilde{\pi}_i(p)$. Thus $p_i\subseteq\widetilde{\pi}_i(p)$, and both being ultrafilters, we have that $p_i=\widetilde{\pi}_i(p)$. Therefore $\widetilde{\pi}_i\big[\overline{e}_{I^\Diamond}^*\big]=e_{\pi_i[I^\Diamond]}^*$.
\end{proof}

\begin{lemma}\label{3.5}
Let $S$ be a near $e$ subsemigroup of a semitopological semigroup $T$. Let $l\in\mathbb{N}$. Let $I^\Diamond\subseteq S^l$ be a near $\overline{e}$ subsemigroup of $T^l$ such that for each $i\in\{1,2,\ldots, l\}$ and for $\mathcal{E}_i\subseteq\pi_i[I^\Diamond]$, if $e\in cl_T(\mathcal{E}_i)$, then $\overline{e}\in cl_{T^l}\big(\pi_i^{-1}(\mathcal{E}_i)\big)$ where $\pi_i:I^\Diamond\rightarrow\pi_i[I^\Diamond]$ is the natural projection homomorphism onto the $i^{th}$ coordinate. Let $A\subseteq I^\Diamond$.
\begin{enumerate}
\item If $A\subseteq I^\Diamond$ is central near $\overline{e}$, then $\pi_i[A]\subseteq\pi_i[I^\Diamond]$ is central near $e$ for each $i\in\{1,2,\ldots,l\}$.
\item If $e$ has a countable local base in $T$ and if $A\subseteq I^\Diamond$ is piecewise syndetic near $\overline{e}$, then $\pi_i[A]\subseteq\pi_i[I^\Diamond]$ is piecewise syndetic near $e$ for each $i\in\{1,2,\ldots,l\}$.
\end{enumerate}
\end{lemma}

\begin{proof}
For each $i\in \{1,2,\ldots,l\}$, let $\widetilde{\pi}_i:\beta I^\Diamond\rightarrow\beta\big(\pi_i[I^\Diamond]\big)$ be the continuous extension of $\pi_i$. By Lemma \ref{Lemma 3.4} and \cite[Lemma 2.14]{hs}, $\widetilde{\pi}_i|_{\overline{e}_{I^\Diamond}^*}:\overline{e}_{I^\Diamond}^*\rightarrow e_{\pi_i[I^\Diamond]}^*$ is a surjective homomorphism. Hence, by \cite[Exercise 1.7.3]{hs} we have that $\widetilde{\pi}_i[K(\overline{e}_{I^\Diamond}^*)]=K\big(e_{\pi_i[I^\Diamond]}^*\big)$.
\begin{enumerate}
\item Pick an idempotent $p\in K(\overline{e}_{I^\Diamond}^*)$ such that $A\in p$. Then $\widetilde{\pi}_i(p)$ is an idempotent in $K\big(e_{\pi_i[I^\Diamond]}^*\big)$ and $\pi_i[A]\in\widetilde{\pi}_i(p)$.
\item Using Lemma \ref{2.18} pick $p\in K(\overline{e}_{I^\Diamond}^*)$ such that $A\in p$. Then $\widetilde{\pi}_i(p)\in K\big(e_{\pi_i[I^\Diamond]}^*\big)$ and $\pi_i[A]\in\widetilde{\pi}_i(p)$.
\end{enumerate}
\end{proof}

\begin{theorem}\label{3.6}
Let $S$ be a near $e$ subsemigroup of a semitopological semigroup $T$ where $e$ has a countable local base in $T$. Let $l\in\mathbb{N}$. Let $I^\Diamond\subseteq S^l$ be a near $\overline{e}$ subsemigroup of $T^l$ such that for each $i\in\{1,2,\ldots,l\}$, $\pi_i[I^\Diamond]\subseteq S$ is piecewise syndetic near $e$, and for $\mathcal{E}_i\subseteq\pi_i[I^\Diamond]$ if $e\in cl_T(\mathcal{E}_i)$, then $\overline{e}\in cl_{T^l}\big(\pi_i^{-1}(\mathcal{E}_i)\big)$ where $\pi_i:I^\Diamond\rightarrow\pi_i[I^\Diamond]$ is the natural projection homomorphism onto the $i^{th}$ coordinate. Let $B\subseteq S$.
\begin{enumerate}
\item If $B\subseteq S$ is $central^*$ near $e$, then $B^l\cap I^\Diamond\subseteq I^\Diamond$ is $central^*$ near $\overline{e}$.
\item If $B\subseteq S$ is $PS^*$ near $e$, then $B^l\cap I^\Diamond\subseteq I^\Diamond$ is piecewise syndetic near $\overline{e}$. 
\end{enumerate}
\end{theorem}
\begin{proof}
Let $i\in\{1,2,\ldots,l\}$ and pick $p\in K(e_S^*)\cap cl_{\beta S}(\pi_i[I^\Diamond])$. Then for each $U\in\tau_e$, $U\cap\pi_i[I^\Diamond]=(U\cap S)\cap \pi_i[I^\Diamond]\in p$, and therefore using the natural inclusion of $e_{\pi_i[I^\Diamond]}^*$ inside $e_S^*$ we have that $p\in e_{\pi_i[I^\Diamond]}^*$. Thus $e_{\pi_i[I^\Diamond]}^*\cap K(e_S^*)\neq\emptyset$ and therefore by \cite[Theorem 1.65]{hs}, we have that $K\big(e_{\pi_i[I^\Diamond]}^*\big)\subseteq K(e_S^*)$.

\begin{enumerate}
\item By Lemma \ref{3.3} it suffices to show that whenever $A\subseteq I^\Diamond$ is central near $\overline{e}$, then $\pi_i[A]\subseteq S$ is central near $e$ for each $i\in\{1,2,\ldots,l\}$. Fix $i\in\{1,2,\ldots,l\}$ and let $A\subseteq I^\Diamond$ be central near $\overline{e}$. Then by Lemma \ref{3.5}, $\pi_i[A]\subseteq\pi_i[I^\Diamond]$ is central near $e$. So pick an idempotent $p\in K\big(e_{\pi_i[I^\Diamond]}^*\big)\cap cl_{\beta(\pi_i[I^\Diamond])}(\pi_i[A])$ and consequently we have that $p\in K(e_S^*)\cap cl_{\beta S}(\pi_i[A])$.
\item By Lemma \ref{3.3} it suffices to show that whenever $A\subseteq I^\Diamond$ is piecewise syndetic near $\overline{e}$, then $\pi_i[A]\subseteq S$ is piecewise syndetic near $e$ for each $i\in\{1,2,\ldots,l\}$. Fix $i\in\{1,2,\ldots,l\}$ and let $A\subseteq I^\Diamond$ be piecewise syndetic near $\overline{e}$. Then by Lemma \ref{3.5}, $\pi_i[A]\subseteq\pi_i[I^\Diamond]$ is piecewise syndetic near $e$. So $K\big(e_{\pi_i[I^\Diamond]}^*\big)\cap cl_{\beta(\pi_i[I^\Diamond])}(\pi_i[A])\neq\emptyset$ and consequently we have that $K(e_S^*)\cap cl_{\beta S}(\pi_i[A])\neq\emptyset$. 
\end{enumerate}
\end{proof}

From now onwards, let $S$ be a near $e$ subsemigroup of a semitopological semigroup $(T,+)$, let $l\in\mathbb{N}$ and let $E^\Diamond\subseteq S^l$ be a near $\overline{e}$ subsemigroup of $T^l$ such that $\{\overline{a}:a\in S\}\subseteq E^\Diamond$. Moreover, let $I^\Diamond$ be a near $\overline{e}$ ideal of $E^\Diamond$ such that for each $i\in\{1,2,\ldots, l\}$ and for $\mathcal{E}_i\subseteq\pi_i[I^\Diamond]$, if $e\in cl_T(\mathcal{E}_i)$, then $\overline{e}\in cl_{T^l}\big(\pi_i^{-1}(\mathcal{E}_i)\big)$ where $\pi_i:I^\Diamond\rightarrow\pi_i[I^\Diamond]$ is the natural projection homomorphism onto the $i^{th}$ coordinate. Note that when we say $I^\Diamond$ is a near $\overline{e}$ ideal of $E^\Diamond$, we mean that $I^\Diamond$ is an ideal of $E^\Diamond$ which is also a near $\overline{e}$ subsemigroup of $T^l$.

\begin{definition}
Let $X=(\beta S)^l, Y=(e_S^*)^l$ are with product topology and the coordinatewise operation. Then $E=cl_XE^\Diamond, I=cl_XI^\Diamond, E_0=E\cap Y$ and $I_0=I\cap Y$.
\end{definition}
\begin{lemma}\label{3.8}
X is a compact right topological semigroup, for each $\vec{x}\in S^l$, $\lambda_{\vec{x}}:X\rightarrow X$ given by $\lambda_{\vec{x}}(\vec{p})=\vec{x}+\vec{p},\,\vec{p}\in X$ is continuous, $Y$ is a subsemigroup of $X$, $E_0$ is a subsemigroup of $Y$, $I_0$ is an ideal of $E_0$, and $K(Y)=(K(e_S^*))^l$.
\end{lemma}
\begin{proof}
It suffices to show that $E_0\neq\emptyset$ and $I_0\neq\emptyset$. Remaining of the proof follows from \cite[Theorem 2.22, 2.23 and 4.17]{hs}. Let $\tilde{\iota}:\beta E^\Diamond\rightarrow X$ be the continuous extension the natural inclusion map $\iota:E^\Diamond\rightarrow X$. Now we prove that $E_0\neq\emptyset$ and the proof for $I_0$ is same as of $E_0$. Pick $p\in\overline{e}_{E^\Diamond}^*$. Following the proof of \cite[Theorem 3.27]{hs} we then have that $\tilde{\iota}(p)\in\bigcap_{A\in p}cl_X (\iota[A])$. In particular, $\tilde{\iota}(p)\in cl_X (\iota[E^\Diamond])=E$. Let $i\in\{1,2,\ldots,l\}$ and let $\pi_i:E^\Diamond\rightarrow \pi_i[E^\Diamond]$ and $proj_i:X\rightarrow\beta S$ be the natural projection maps onto the $i^{th}$ coordinates of $E^\Diamond$ and $X$, respectively. If $\widetilde{\pi}_i:\beta(E^\Diamond)\rightarrow\beta (\pi_i[E^\Diamond])\subseteq\beta S$ is the continuous extension of $\pi_i$, then $proj_i\circ\tilde{\iota}=\widetilde{\pi}_i$ as both the sides agree on a dense set $E^\Diamond$ of $\beta E^\Diamond$. Now, for each $U\in\tau_e$ we have that $U^l\cap E^\Diamond\in p$, which implies $\pi_i[U^l\cap E^\Diamond]\in\widetilde{\pi}_i(p)=proj_i(\tilde{\iota}(p))$, i.e. $U\cap\pi_i[E^\Diamond]\in proj_i(\tilde{\iota}(p))$, and therefore $U\cap S\in proj_i(\tilde{\iota}(p))$. Hence $proj_i(\tilde{\iota}(p))\in e_S^*$ and thus $\tilde{\iota}(p)\in E\cap Y=E_0$. 
\end{proof}
\begin{lemma}\label{3.9}
Let $p\in K(e_S^*)$. Then $\overline{p}\in K(I_0)=(K(e_S^*))^l\cap {E_0}$.
\end{lemma}
\begin{proof}
Observe first that $\overline{p}\in E$. To see this let $U$ be a neighborhood of $\overline{p}$ in $X$ and for each $i\in\{1,2,\ldots,l\}$, pick $A_i\in p$ such that $\bigtimes_{i=1}^lcl_{\beta S}(A_i)\subseteq U$. Then $\bigcap_{i=1}^l A_i\in q$ so pick $a\in\bigcap_{i=1}^l A_i$. Then $\overline{a}\in E^\Diamond\cap U$.

Thus $\overline{p}\in (K(e_S^*))^l\cap E=(K(e_S^*))^l\cap E_0$. Since, by Lemma \ref{3.8}, $K(Y)=(K(e_S^*))^l$ we thus have that $K(Y)\cap E_0\neq\emptyset$. Thus, by \cite[Theorem 1.65]{hs} $K(E_0)=K(Y)\cap E_0=(K(e_S^*))^l\cap E_0$. 

Since $I_0$ is an ideal of $E_0$, $K(E_0)\subseteq I_0$. Consequently, again by \cite[Theorem 1.65]{hs}, $K(I_0)=K(E_0)=(K(e_S^*))^l\cap E_0$. Thus $\overline{p}\in K(I_0)$.
\end{proof}
Now $I^\Diamond$ itself is a near $\overline{e}$ subsemigroup of $T^l$, and thus $\beta I^\Diamond$ is a compact right topological semigroup.

\begin{definition}
$\iota:I^\Diamond\rightarrow I^\Diamond\subseteq I$ is the identity function and $\tilde{\iota}:\beta I^\Diamond\rightarrow I$ is its continuous extension.
\end{definition}

\begin{lemma}\label{3.12}
The function $\tilde{\iota}$ is a homomorphism and $\tilde{\iota}[K(\overline{e}_{I^\Diamond}^*)]=K(I_0)$.
\end{lemma}
\begin{proof}
Let $i\in\{1,2,\ldots,l\}$. Let $proj_i:X\rightarrow\beta S$ be the natural projection map onto the $i^{th}$ coordinates of $X$. If $\widetilde{\pi}_i:\beta I^\Diamond\rightarrow\beta(\pi_i[I^\Diamond])\subseteq\beta S$ is the continuous extension of $\pi_i$, then $\widetilde{\pi}_i=proj_i\circ\tilde{\iota}$ as both sides agree on a dense set $I^\Diamond$ of $\beta I^\Diamond$.

By \cite[Lemma 2.14]{hs}, $\tilde{\iota}$ is a homomorphism. Now we show that $\tilde{\iota}[\overline{e}_{I^\Diamond}^*]=I_0$. To see this, first let $p\in\overline{e}_{I^\Diamond}^*$. Then $\tilde{\iota}(p)\in I$, $proj_i(\tilde{\iota}(p))=\widetilde{\pi}_i(p)$ and by Lemma \ref{Lemma 3.4} we have that $\widetilde{\pi}_i(p)\in e_{\pi_i[I^\Diamond]}^*\subseteq e_S^*$. Hence $proj_i(\tilde{\iota}(p))\in e_S^*$ and thus $\tilde{\iota}(p)\in Y$.

Conversely, let $(p_1,p_2,\ldots,p_l)\in I_0$. For each $j\in\{1,2,\ldots,l\}$, if $A_j\in p_j$, then $\bigtimes_{j=1}^l cl_{\beta S}(A_j)$ is a neighborhood of $(p_1,p_2,\ldots,p_l)$ in $X$ and therefore $\big(\bigtimes_{j=1}^l A_j\big)\cap I^\Diamond=\big(\bigtimes_{j=1}^l cl_{\beta S}(A_j)\big)\cap I^\Diamond\neq\emptyset$. Moreover, the family $\big\{\big(\bigtimes_{j=1}^l A_j\big)\cap I^\Diamond: A_j\in p_j \text{ for each } j\in\{1,2,\ldots,l\}\big\}$ has finite intersection property, hence pick $p\in\beta I^\Diamond$ such that $\big\{\big(\bigtimes_{j=1}^l A_j\big)\cap I^\Diamond: A_j\in p_j \text{ for each } j\in\{1,2,\ldots,l\}\big\}\subseteq p$. We claim that $\tilde{\iota}(p)=(p_1,p_2,\ldots,p_l)$ for which it suffices to show that $proj_i\circ\tilde{\iota}(p)=p_i$, i.e. $\widetilde{\pi}_i(p)=p_i$. To prove the claim, let $B_i\in p_i$. Then $\pi_i^{-1}(B_i)=\big(\bigtimes_{j=1}^l B_j\big)\cap I^\Diamond\in p$, where $B_j:=S$ for each $j\neq i$, and therefore $B_i\in\widetilde{\pi}_i(p)$. Hence $p_i\subseteq\widetilde{\pi}_i(p)$, and both being ultrafilters, we have that $p_i=\widetilde{\pi}_i(p)$. Now we show that $p\in\overline{e}_{I^\Diamond}^*$. For any $U\in\tau_{\overline{e}}$ pick $V\in\tau_e$ such that $V^l\subseteq U$. Then $V\cap S\in p_j$ for each $j\in\{1,2,\ldots,l\}$, so $V^l\cap I^\Diamond=(V\cap S)^l\cap I^\Diamond\in p$ and therefore $U\cap I^\Diamond\in p$. Thus $p\in\overline{e}_{I^\Diamond}^*$. 

Therefore $\tilde{\iota}|_{\overline{e}_{I^\Diamond}^*}: {\overline{e}_{I^\Diamond}^*}\rightarrow I_0$ is a surjective homomorphism. Now, by \cite[Exercise 1.7.3]{hs}, we have that $\tilde{\iota}[K(\overline{e}_{I^\Diamond}^*)]=K(I_0)$.
\end{proof}

Now we recall \cite[Lemma 3.6]{bh} which will be used subsequently.
\begin{lemma}\label{3.11}
Let $B\subseteq S$. If $r\in \beta I^\Diamond$ and $\tilde{\iota}(r)\in (cl_{\beta S}B)^l$. Then $B^l\cap I^\Diamond \in r$.
\end{lemma}

\begin{theorem}\label{3.13}
Let $S$ be a near $e$ subsemigroup of a semitopological semigroup $(T,+)$, let $l\in\mathbb{N}$ and let $E^\Diamond\subseteq S^l$ be a near $\overline{e}$ subsemigroup of $T^l$ such that $\{\overline{a}:a\in S\}\subseteq E^\Diamond$. Let $I^\Diamond$ be a near $\overline{e}$ ideal of $E^\Diamond$ such that for each $i\in\{1,2,\ldots, l\}$ and for $\mathcal{E}_i\subseteq\pi_i[I^\Diamond]$, if $e\in cl_T(\mathcal{E}_i)$, then $\overline{e}\in cl_{T^l}\big(\pi_i^{-1}(\mathcal{E}_i)\big)$ where $\pi_i:I^\Diamond\rightarrow\pi_i[I^\Diamond]$ is the natural projection homomorphism onto the $i^{th}$ coordinate. Let $B\subseteq S$. 

\begin{enumerate}
\item If $B\subseteq S$ is central near $e$, then $B^l\cap I^\Diamond\subseteq I^\Diamond$ is central near $\overline{e}$.
\item Let $e$ have a countable local base in $T$. If $B\subseteq S$ is piecewise syndetic near $e$, then $B^l\cap I^\Diamond\subseteq I^\Diamond$ is piecewise syndetic near $\overline{e}$.
\item\label{3.13(3)} Let $e$ have a countable local base in $T$. If $B\subseteq S$ is $central^*$ near $e$, then $B^l\cap I^\Diamond\subseteq I^\Diamond$ is $central^*$ near $\overline{e}$.
\item\label{3.13(4)} Let $e$ have a countable local base in $T$. If $B\subseteq S$ is $PS^*$ near $e$, then $B^l\cap I^\Diamond\subseteq I^\Diamond$ is $PS^*$ near $\overline{e}$.
\item\label{3.13(5)} If $B\subseteq S$ is $IP^*$ near $e$, then $B^l\cap I^\Diamond\subseteq I^\Diamond$ is $IP^*$ near $\overline{e}$.
\end{enumerate}
\end{theorem}
\begin{proof}
\begin{enumerate}
\item Pick $p\in K(e_S^*)$ such that $p+p=p$ and $B\in p$. By Lemma \ref{3.9}, $\overline{p}\in K(I_0)$. Pick by Lemma \ref{3.12}  $q\in K(\overline{e}_{I^\Diamond}^*)$ such that $\tilde{\iota}(q)=\overline{p}$. Pick a minimal left ideal $L$ of $\overline{e}_{I^\Diamond}^*$ such that $q\in L$. Let $M=\{r\in L: \tilde{\iota}(r)=\overline{p}\}$. Then $M$ is a compact subsemigroup of $\overline{e}_{I^\Diamond}^*$, so pick an idempotent $r\in M$. By Lemma \ref{3.11}, $B^l\cap I^\Diamond\in r$.
\item Pick by Lemma \ref{2.18} $p\in K(e_S^*)$ such that $B\in p$. Now Lemma \ref{3.9} provides that $\overline{p}\in K(I_0)$. Pick by Lemma \ref{3.12} $q\in K(\overline{e}_{I^\Diamond}^*)$ such that $\tilde{\iota}(q)=\overline{p}$. Hence, by Lemma \ref{3.11}, we have $B^l\cap I^\Diamond \in r$.
\end{enumerate}\

To establish statements (\ref{3.13(3)}) and (\ref{3.13(4)}), it suffices by Theorem \ref{3.6} to let $i\in\{1,2,\ldots,l\}$ and show that $\pi_i[I^\Diamond]\subseteq S$ is piecewise syndetic near $e$. Since $I^\Diamond$ is a near $\overline{e}$ ideal of $E^\Diamond$, so in particular $\pi_i[I^\Diamond]$ is a left ideal of $\pi_i[E^\Diamond]=S$ and therefore $cl_{\beta S}(\pi_i[I^\Diamond])$ is a left ideal of $\beta S$. Moreover, as $e\in cl_T(\pi_i[I^\Diamond])$, so $cl_{\beta S}(\pi_i[I^\Diamond])\cap e_S^*\neq\emptyset$. Hence $cl_{\beta S}(\pi_i[I^\Diamond])\cap e_S^*$ is a left ideal of $e_S^*$. Pick a minimal left ideal $L$ of $e_S^*$ contained in $cl_{\beta S}(\pi_i[I^\Diamond])\cap e_S^*$. Then $L\cap cl_{\beta S}(\pi_i[I^\Diamond])=L\neq\emptyset$ and therefore $K(e_S^*)\cap cl_{\beta S}(\pi_i[I^\Diamond])\neq\emptyset$. Thus, by Lemma \ref{2.18}, $\pi_i[I^\Diamond]\subseteq S$ is piecewise syndetic near $e$. Statement (\ref{3.13(5)}) follows immediately from Corollary \ref{3.4}.
\end{proof}
\begin{corollary}\label{Corollary 3.13}
Let $S\subseteq (0,\infty)$ be a near zero subsemigroup of $([0, \infty), +)$. Let $l\in \mathbb{N}$ and let $AP_l=\{(a, a+d, \ldots,a+(l-1)d):a, d\in S\}$. Let "large near zero" be any of  "central near zero", "piecewise syndetic near zero", "$central^*$ near zero", "$PS^*$ near zero" or "$IP^*$ near zero". Let $B\subseteq S$. If $B\subseteq S$ is large near zero, then $B^l\cap AP_l\subseteq AP_l$ is large near zero.
\end{corollary}
\begin{proof}
Let $T=([0, \infty), +), \,I^\Diamond=AP_l$ and let $E^\Diamond=I^\Diamond \cup \{\overline{a}:a\in S\}$. It suffices by Theorem \ref{3.13} to let $i\in\{1,2,\ldots,l\}$ and show that if $\mathcal{E}_i\subseteq\pi_i[AP_l]$ and if $0\in cl_{[0,\infty)}(\mathcal{E}_i)$, then $\overline{0}\in cl_{[0,\infty)^l}\big(\pi_i^{-1}(\mathcal{E}_i)\big)$. Note that $\pi_i[AP_l]=\{a+(i-1)d:a,d\in S\}$ and $\pi_i^{-1}(\mathcal{E}_i)=\{(a,a+d,\ldots,a+(l-1)d)\in AP_l:a,d\in S \text{ and } a+(i-1)d\in\mathcal{E}_i\}$. To show $\overline{0}\in cl_{[0,\infty)^l}\big(\pi_i^{-1}(\mathcal{E}_i)\big)$, it suffices to let $\epsilon>0$ and show that $[0,\epsilon)^l\cap\pi_i^{-1}(\mathcal{E}_i)\neq\emptyset$. Since $0\in cl_{[0,\infty)}(\mathcal{E}_i)$, so pick $a,d\in S$ such that $0<a+(i-1)d<\epsilon/l$. Then for each $j\in\{1,2,\ldots,l\}$, $0<a+(j-1)d<\epsilon$, and therefore $(a,a+d,\ldots,a+(l-1)d)\in [0,\epsilon)^l\cap\pi_i^{-1}(\mathcal{E}_i)$
\end{proof}
\begin{theorem}\label{Theorem 3.14}
Let $S\subseteq (0,\infty)$ be a near zero subsemigroup of $([0, \infty), +)$. Let $l\in \mathbb{N}$ and let $AP_l=\{(a, a+d, \ldots, a+(l-1)d):a, d\in S\}$. Let "large near zero" be any of  "central near zero", "piecewise syndetic near zero", "$central^*$ near zero", "$PS^*$ near zero" or "$IP^*$ near zero". Let $B\subseteq S$. If $B\subseteq S$ be large near zero, then  $\{(a, d)\in S\times S : \{a, a+d, \ldots, a+(l-1)d\}\subseteq B\}\subseteq S\times S$ is large near zero.
\end{theorem}
\begin{proof}
Let $AP_l=\{(a, a+d, a+2d, \ldots, a+(l-1)d): a, d\in S\}$. Then the function $\psi:(S\times S)\cup\{(0,0)\}\rightarrow AP_l\cup\{\overline{0}\}$ defined by $\psi(0,0)=\overline{0}$ and $\psi(a, d)=(a, a+d, \ldots, a+(l-1)d)$, $a,d\in S$ is an isomorphism. If $B\subseteq S$ is large near zero, then by Corollary \ref{Corollary 3.13}, $B^l\cap AP_l\subseteq AP_l$ is large near zero, hence $\psi^{-1}(B^l\cap AP_l)\subseteq \psi^{-1}(AP_l)$is large near zero, i.e. $\{(a,d)\in S\times S:\{a,a+d,\ldots,a+(l-1)d\}\subseteq B\}\subseteq S\times S$ is large near zero. 
\end{proof}
Now we recall the notion of minimal idempotent \cite[Definition 1.37]{hs} which will be used subsequently. 
\begin{definition}
Let $(\mathcal{S},+)$ be a semigroup and let $e,f\in E(\mathcal{S})$. Then $e\leq f$ if $e=e+f=f+e$.
\end{definition}
\begin{definition}
Let $\mathcal{S}$ be a semigroup. Let $e\in E(\mathcal{S})$. Then $e$ is a minimal idempotent if $f\in E(\mathcal{S})$ and $f\leq e$ implies $f=e$.
\end{definition}
\begin{lemma}\label{3.16}
Let $u, v\in \mathbb{N}$ and let $A$ be an $u\times v$ matrix with entries from $\mathbb{N}\cup \{0\}$ such that no row of $A$ is zero. Let $S\subseteq (0,\infty)$ be a dense subsemigroup of $([0, \infty), +)$. Let $\widetilde{\phi}:\beta(S^v)\rightarrow (\beta S)^u$ be the continuous extension of the map $\phi:S^v\rightarrow S^u\subseteq (\beta S)^u$ defined by $\phi(\vec{x})=A\vec{x}$, $\vec{x}\in S^v$. Let $p$ be a minimal idempotent of $0_S^*$ and assume that for every $C\in p$ there exists $\vec{x}\in S^v$ such that $A\vec{x}\in C^u$. Then there is a minimal idempotent $q\in \overline{0}_{S^v}^*$ such that $\widetilde{\phi}(q)=\overline{p}$.
\end{lemma}
\begin{proof}
First, we claim that $\overline{p}\in \widetilde{\phi}\big[\overline{0}_{S^v}^*\big]$. Suppose that $\overline{p}\in (\beta S)^u \setminus \widetilde{\phi}\big[\overline{0}_{S^v}^*\big]$. Since $\widetilde{\phi}\big[0_{S^v}^*\big]$ is closed, so pick a neighborhood $U$ of $\overline{p}$ such that $U\cap \widetilde{\phi}\big[\overline{0}_{S^v}^*\big]= \emptyset$. Pick $C\in p$ such that $(cl_{\beta S}C)^u\subseteq U$. Let $\mathcal{X}=\{\vec{x}\in S^v:A\vec{x}\in C^u\}$. Then $\{(S\cap (0, \epsilon))^v:\epsilon>0\}\cup \{\mathcal{X}\}$ has finite intersection property, and so pick $r\in \overline{0}_{S^v}^*$ such that
$\{(S\cup(0, \epsilon))^v:\epsilon>0\}\cup\{\mathcal{X}\}\subseteq r$. Then $\widetilde{\phi}(r) \in\widetilde{\phi}[0_{S^v}^*]$. Moreover, $\phi[\mathcal{X}]\in\widetilde{\phi}(r)$ and $\phi[\mathcal{X}]\subseteq C^u$ implies $C^u\in\widetilde{\phi}(r)$, i.e. $\widetilde{\phi}(r)\in (cl_{\beta S}C)^u$ and thus $\widetilde{\phi}(r)\in U\cap\widetilde{\phi}\big[\overline{0}_{S^v}^*\big]$, a contradiction. 

Let $M=\{q\in \overline{0}_{S^v}^*:\widetilde{\phi}(q)=\overline{p}\}$. Then $M$ is a compact subsemigroup of $0_{S^v}^*$, so pick an idempotent $w\in M$. Pick by \cite[Theorem 1.60]{hs} a minimal idempotent $q\in \overline{0}_{S^v}^*$ such that $q\leq w$. Since $\widetilde{\phi}$
is a homomorphism, we have $\widetilde{\phi}(q)\leq \widetilde{\phi}(w)=\overline{p}$ so, since $\overline{p}$ is  minimal in $(0_S^*)^u$, we have that $\widetilde{\phi}(q)=\overline{p}$.
\end{proof}
\begin{theorem}\label{3.17}
Let $u, v\in \mathbb{N}$ and let $A$ be an $u\times v$ matrix with entries from $\mathbb{N}\cup \{0\}$ such that no row of $A$ is zero. Let $S\subseteq (0,\infty)$ be a dense subsemigroup of $([0, \infty), +)$. Assume that for every central set near zero $C$ in $S$, there exists $\vec{x}\in S^v$ such that $A\vec{x}\in C^u$. Then for every central set near zero $C$ in $S$, $\{\vec{x}\in S^v: A\vec{x}\in C^u\}$ is central near zero in $S^v$.
\end{theorem}
\begin{proof}
Let $C\subseteq S$ be a central set near zero. Pick an idempotent $p\in K(0_S^*)$ such that $C\in p$. By \cite[Theorem 1.59]{hs}, $p$ is a minimal idempotent in $0_S^*$. Let $\widetilde{\phi}:\beta(S^v)\rightarrow (\beta S)^u$ be the continuous extension of the map $\phi:S^v\rightarrow S^u\subseteq (\beta S)^u$ defined by $\phi(\vec{x})=A\vec{x}$, $\vec{x}\in S^v$. Pick by Lemma \ref{3.16} a minimal idempotent $q\in\overline{0}_{S^v}^*$ such that $\widetilde{\phi}(q)=\overline{p}$. By \cite[Theorem 1.59]{hs}, $q$ is an idempotent in $K\big(\overline{0}_{S^v}^*\big)$. Since $\bigtimes_{i=1}^u cl_{\beta S} C$ is a neighborhood of $\overline{p}$, so pick $B\in q$ such that $\widetilde{\phi}\big[cl_{\beta(S^v)}B\big]\subseteq\bigtimes_{i=1}^u cl_{\beta S} C$. Then $B\subseteq \{\vec{x}\in S^v: A\vec{x}\in C^u\}$, so  $\{\vec{x}\in S^v:A\vec{x}\in C^u\}\in q$ as required. 
\end{proof}
\begin{theorem}\label{3.19}
Let $u, v\in \mathbb{N}$ and $A$ be an $u\times v$ matrix with entries from $\mathbb{N}\cup\{0\}$ such that none of the rows of $A$ is zero. Let $S\subseteq (0,\infty)$ be a dense subsemigroup of $([0, \infty), +)$. If
\begin{enumerate}
    \item for each central set near zero $C$ in $S$, there exists $\vec{x}\in S^v$ such that $A\vec{x}\in C^u$ and
    \item\label{3.19(2)} for each $\epsilon>0$ and for each $s\in S\cap(0,\epsilon)$, there exists $\vec{x}\in (S\cap (0, \epsilon))^v$ such that $A\vec{x}=\overline{s}$,
\end{enumerate}
then for every piecewise syndetic set near zero $B$ in $S, \, \{\vec{x}\in S^v: A\vec{x}\in B^u\}\subseteq S^v$ is piecewise syndetic near zero.
\end{theorem}
\begin{proof}
Let $B\subseteq S$ be piecewise syndetic near zero. Pick by \cite[Theorem 2.5]{bcp} $p\in E(K(0_S^*))$ such that for each $\epsilon>0$ there exists $s\in S\cap (0,\epsilon)$ for which $-s+B\in p$. Let $\widetilde{\phi}:\beta(S^v)\rightarrow (\beta S)^u$ be the continuous extension of the map $\phi:S^v\rightarrow S^u\subseteq (\beta S)^u$ defined by $\phi(\vec{x})=A\vec{x}$, $\vec{x}\in S^v$. Since $p$ is a minimal idempotent in $0_S^*$ \cite[Theorem 1.59]{hs}, pick by Lemma \ref{3.16} $q\in E\big(K\big(\overline{0}_{S^v}^*\big)\big)$ such that $\widetilde{\phi}(q)=\overline{p}$. To show that $\{\vec{x}\in S^v:A\vec{x}\in B^u\}=\phi^{-1}[B^u]\subseteq S^v$ is piecewise syndetic near zero, it suffices by \cite[Theorem 2.5]{bcp} to let $\epsilon>0$ and find $\vec{x}\in(S\cap (0,\epsilon))^v$ such that $-\vec{x}+\phi^{-1}[B^u]\in q$. Pick $s\in S\cap (0,\epsilon)$ such that $-s+B\in p$. Then $(-s+B)^u\in\overline{p}=\widetilde{\phi}(q)$, i.e. $\phi^{-1}[(-s+B)^u]\in q$. Pick by Assumption (\ref{3.19(2)}) $\vec{x}\in (S\cap (0,\epsilon))^v$ such that $A\vec{x}=\overline{s}$. Then
\begin{align*}
\phi^{-1}[(-s+B)^u]&\subseteq \phi^{-1}[-\overline{s}+B^u]\\
&\subseteq \{\vec{y}\in S^v:\overline{s}+A\vec{y}\in B^u\}\\
&=\{\vec{y}\in S^v:\phi(\vec{x}+\vec{y})\in B^u\}\\
&\subseteq\{\vec{y}\in S^v:\vec{x}+\vec{y}\in\phi^{-1}[B^u]\}=-\vec{x}+\phi^{-1}[B^u],
\end{align*}
and therefore $-\vec{x}+\phi^{-1}[B^u]\in q$.
\end{proof}
\section{Tensor product near zero}\label{4}
Let us begin by recalling the notion of tensor product \cite[Definition 1.1]{hs1}.

\begin{definition}
Let $X$ and $Y$ be discrete spaces, let $p\in \beta X$, and $q\in\beta Y$. Then the tensor product of $p$ and $q$ is $p\otimes q=\{A\subseteq X\times Y:\{x\in X:\{y\in Y:(x,y)\in A\}\in q\}\in p\}$. 
\end{definition}

It is easy to see that $A\times B\in p\otimes q$ if and only if $A\in p$ and $B\in q$.

\begin{theorem}
Let $S_1$ and $S_2$ be two discrete subsemigroups, and let $\tilde{\iota}:\beta (S_1\times S_2)\rightarrow \beta S_1\times \beta S_2$ be the continuous extension of the inclusion map $\iota:(S_1\times S_2)_d\rightarrow \beta S_1\times \beta S_2$. Let $p_1\in\beta S_1$ and $p_2\in\beta S_2$. Then $\tilde{\iota}(p_1\otimes p_2)=(p_1, p_2)$.
\end{theorem}
\begin{proof}
For each $i\in\{1,2\}$, if $\pi_i:S_1\times S_2\rightarrow S_i$ is the natural projection homomorphism onto the $i^{th}$ coordinate and if $\widetilde{\pi}_i:\beta(S_1\times S_2)\rightarrow\beta S_i$ is the continuous extension of $\pi_i$, then $\tilde{\iota}(p_1\otimes p_2)=(\widetilde{\pi}_1(p_1\otimes p_2),\widetilde{\pi}_2(p_1\otimes p_2))$. Now $\widetilde{\pi}_1(p_1\otimes p_2)=\{A_1\subseteq S_1:\pi_1^{-1}(A_1)\in p_1\otimes p_2\}=\{A_1\subseteq S_1:A_1\times S_2\in p_1\otimes p_2\}=p_1$. The proof of $\widetilde{\pi}_2(p_1\otimes p_2)=p_2$ is exactly similar. Thus $\tilde{\iota}(p_1\otimes p_2)=(p_1, p_2)$.
\end{proof}
The above theorem now provides us the following version near idempotent.
\begin{corollary}
For each $i\in\{1,2\}$, let $S_i$ be a near $e_i$ subsemigroup of a semitopological semigroup $T_i$. Let $(e_1)_{S_1}^*\otimes (e_2)_{S_2}^*=\{p\otimes q:p\in (e_1)_{S_1}^*, \,q\in (e_2)_{S_2}^*\}$. Let $\tilde{\iota}:\beta (S_1\times S_2)\rightarrow \beta S_1\times \beta S_2$ be the continuous extension of the inclusion map $\iota:(S_1\times S_2)_d\rightarrow \beta S_1\times \beta S_2$. Then $\tilde{\iota}[(e_1)_{S_1}^*\otimes (e_2)_{S_2}^*]=(e_1)_{S_1}^*\times (e_2)_{S_2}^*$.
\end{corollary}

We now recall the definition of higher dimensional tensor product of ultrafilters \cite[definition 1.15]{bhw}.
\begin{definition}\label{4.4}
Let $l\in \mathbb{N}$, and for each $i\in \{1,2, \ldots,l\}$ let $S_i$ be a semigroup and let $p_i\in \beta S_i$. We define $\otimes_{i=1}^l p_i\in \beta\big(\bigtimes_{i=1}^lS_i\big)$ inductively as follows.
\begin{enumerate}
    \item  $\otimes_{i=1}^1p_i=p_1$,
    \item Given $l\in \mathbb{N}$ and $A\subseteq \bigtimes_{i=1}^{l+1} S_i$, $A\in \otimes_{i=1}^{l+1}p_i$ if $\{(x_1, x_2, \ldots, x_l)\in \bigtimes_{i=1}^l S_i:\{x_{l+1}\in S_{l+1}:(x_1, x_2, \ldots, x_{l+1})\in A\}\in p_{l+1}\}\in \otimes_{i=1}^l p_i$.
\end{enumerate}
\end{definition}
\begin{lemma}\label{4.5}
Let $l\in\mathbb{N}$. For each $i\in\{1,2,\ldots,l\}$ let $S_i\subseteq (0,\infty)$ be a dense subsemigroup of $([0,\infty),+)$ and let $p_i\in 0_{S_i}^*$. Let us consider the product semigroup $S=\bigtimes_{i=1}^lS_i$. Then $\otimes_{i=1}^l p_i\in\overline{0}_S^*$.
\end{lemma}
\begin{proof}
It suffices to show that for each $U\in\tau_0$, $U^l\cap S\in\otimes_{i=1}^lp_i$ which follows directly from Definition \ref{4.4}.
\end{proof}
We now recall \cite[Corollary 2.8 and Lemma 2.9]{bhw}.
\begin{lemma}\label{4.6}
Let $m, l\in \mathbb{N}$. For each $i\in \{1, 2, \ldots, l\}$, let $(S_i,+)$ be a semigroup,  $\langle x_{i, n}\rangle_{n=1}^{\infty}$ be a sequence in $S_i$ and  let $p_i$ be an idempotent in $\bigcap_{r=1}^{\infty} cl_{\beta S_i}\big(FS(\langle x_{i, n}\rangle_{n=r}^\infty)\big)$. Let $f:\{1, 2, \ldots, m\}\rightarrow \{1, 2, \ldots l \}$ be a function and let $A\in \otimes_{j=1}^m p_{f(j)}$. Then for each $i\in \{1, 2, \ldots, l\}$ there is a sum subsystem $FS\big(\langle y_{i, n}\rangle_{n=1}^{\infty}\big)$ of $FS\big(\langle x_{i, n}\rangle_{n=1}^\infty\big)$ such that 
$\big\{\big(\sum_{n\in F_1}y_{f(1), n}, \sum_{n\in F_2}y_{f(2), n},$ $\ldots, \sum_{n\in F_m}y_{f(m), n}\big): \text{each } F_j\in\mathcal{P}_f(\mathbb{N}) \text{ and } F_1<F_2<\ldots<F_m\big\}\subseteq A$.
\end{lemma}
\begin{lemma}\label{4.7}
Let $m, l\in \mathbb{N}$. For each $i\in \{1, 2, \ldots, l\}$, let $(S_i,+)$ be a semigroup, let $\langle x_{i, n}\rangle_{n=1}^{\infty}$ be a sequence in $S_i$ and let $p_i\in \bigcap_{r=1}^{\infty} cl_{\beta S_i}\big(FS(\langle x_{i, n}\rangle_{n=r}^\infty)\big)$. Then for any function $f:\{1, 2, \ldots, m\}\rightarrow \{1, 2, \ldots, l\}$, $\big\{\big(\sum_{n\in F_1}x_{f(1), n},$ $\sum_{n\in F_2}x_{f(2), n},\ldots, \sum_{n\in F_m}x_{f(m), n}\big): \text{each } F_j\in\mathcal{P}_f(\mathbb{N}) \text{ and } F_1<F_2<\ldots<F_m\big\}\in \otimes_{j=1}^mp_{f(j)}$.
\end{lemma}
\begin{definition}
Let $m\in\mathbb{N}$ and let $(S,+)$ be a semigroup. Let $\langle a_j\rangle_{j=1}^m$ and $\langle x_n\rangle_{n=1}^\infty$ be sequences in $\mathbb{N}$ and $S$, respectively. The Milliken-Taylor System determined by $\langle a_j\rangle_{j=1}^m$ and $\langle x_n\rangle_{n=1}^\infty$ is $MT\big(\langle a_j\rangle_{j=1}^m,\langle x_n\rangle_{n=1}^\infty\big)=\big\{\sum_{j=1}^m a_j\sum_{n\in F_j}x_n:\text{each } F_j\in\mathcal{P}_f(\mathbb{N}) \text{ and } F_1<F_2<\cdots<F_m\big\}$.
\end{definition}
Regarding the following discussion, note that if $p\in\beta ([0,\infty))$ and $a\in\mathbb{N}$, then $ap$ is the product in $\beta([0,\infty))$ and not the sum of $p$ with itself $a$ times. It is not true in general that $a_1p+a_2p=(a_1+a_2)p$.
\begin{theorem}\label{Theorem 4.9}
Let $m\in\mathbb{N}$ and let $S\subseteq (0,\infty)$ be a dense subsemigroup of $([0,\infty),+)$. Let $\langle a_j\rangle_{j=1}^m$ be a sequence in $\mathbb{N}$ and let $\langle x_n\rangle_{n=1}^\infty$ be a sequence in $S$ such that $\lim_{n\rightarrow\infty}x_n=0$. Let $g(z)=\sum_{j=1}^m a_jz$ and let $A\subseteq S$. The following statements are equivalent.
\begin{enumerate}
\item\label{4.9(1)} There is an idempotent $p\in\bigcap_{r=1}^\infty cl_{\beta S}\big(FS(\langle x_n\rangle_{n=r}^\infty)\big)$ such that $A\in g(p)$. 
\item\label{4.9(2)} For each $\epsilon>0$, there is a sum subsystem $FS(\langle y_n\rangle_{n=1}^\infty)$ of $FS(\langle x_n\rangle_{n=1}^\infty)$ such that $MT\big(\langle a_j\rangle_{j=1}^m,\langle y_n\rangle_{n=1}^\infty\big)\subseteq A\cap (0,\epsilon)$.
\end{enumerate}
\end{theorem}
\begin{proof}
(\ref{4.9(1)}) implies (\ref{4.9(2)}). Let $p$ be an idempotent in $\bigcap_{r=1}^\infty cl_{\beta S}\big(FS(\langle x_n\rangle_{n=r}^\infty)\big)$ such that $A\in g(p)$. By passing to a subsequence, we may presume that $\sum_{n=1}^\infty x_n$ converges. Note that since $\sum_{n=1}^\infty x_n$ converges, $p\in 0_S^*$. Moreover, by using the definition of product it is very easy to see that for each $j\in\{1,2,\ldots,m\}$ and for each $\epsilon>0$ we have that $(0,\epsilon)\cap S\in a_jp$, i.e. $a_jp\in 0_S^*$. Thus $g(p)\in 0_S^*$. So it suffices to show that whenever $Q\in g(p)$, there is a sum subsystem $FS\big(\langle y_n\rangle_{n=1}^\infty\big)$ of $FS\big(\langle x_n\rangle_{n=1}^\infty\big)$ such that $MT\big(\langle a_j\rangle_{j=1}^m,\langle y_n\rangle_{n=1}^\infty\big)\subseteq Q$. The rest of the part now follows from the proof of \cite[Theorem 5.7]{dh}.

(\ref{4.9(2)}) implies (\ref{4.9(1)}). Fix $\epsilon>0$ and pick a sum subsystem $FS\big(\langle y_n\rangle_{n=1}^\infty\big)$ of $FS\big(\langle x_n\rangle_{n=1}^\infty\big)$ such that $MT\big(\langle a_j\rangle_{j=1}^m,\langle y_n\rangle_{n=1}^\infty\big)\subseteq A\cap (0,\epsilon)$. Pick by \cite[Lemma 5.11]{hs} an idempotent $p\in \bigcap_{r=1}^\infty cl_{\beta S}\big(FS\big(\langle y_n\rangle_{n=r}^\infty\big)\big)\subseteq\bigcap_{r=1}^\infty cl_{\beta S}\big(FS\big(\langle x_n\rangle_{n=r}^\infty)\big)$. We claim that $MT\big(\langle a_j\rangle_{j=1}^m,\langle y_n\rangle_{n=1}^\infty\big)\in g(p)$, for which it suffices to show by downward induction on $k\in\{1,2,\ldots,m\}$ that for each $r\in\mathbb{N}$, $\big\{\sum_{j=k}^m a_j\sum_{n\in F_j}y_t: \text{each } F_j\in\mathcal{P}_f(\{r,r+1,\ldots\}) \text{ and } F_k<F_{k+1}<\cdots <F_m\big\}\in a_kp+a_{k+1}p+\cdots +a_mp$.

For $k=m$, we have $FS\big(\langle y_n\rangle_{n=r}^\infty\big)\in p$ so that $FS\big(\langle a_m y_n\rangle_{n=r}^\infty\big)\in a_m p$. So let $k\in\{1,2,\ldots,m-1\}$ and assume that the assertion is true for $k+1$. Let $r\in\mathbb{N}$ and let $B=\big\{\sum_{j=k}^ma_j\sum_{n\in F_j}y_n: \text{each } F_j\in\mathcal{P}_f(\{r,r+1,\ldots\}) \text{ and } F_k<F_{k+1}<\cdots<F_m\big\}$. We show that $FS\big(\langle a_ky_n\rangle_{n=r}^\infty\big)\subseteq \{y\in S:-y+B\in a_{k+1}p+a_{k+2}p+\cdots +a_mp\}$.

So let $b\in FS\big(\langle a_ky_n\rangle_{n=r}^\infty\big)$ and pick $F_k\in\mathcal{P}_f(\{r,r+1,\ldots\})$ such that $b=a_k\sum_{n\in F_k}y_n$. Let $t=\max F_k +1$. Then $\big\{\sum_{j=k+1}^ma_j\sum_{n\in F_j}y_n: \text{each } F_j\in\mathcal{P}_f(\{t,t+1,\ldots\}) \text{ and } F_{k+1}<F_{k+2}<\cdots <F_m\big\}\subseteq -b+B$, so $-b+B\in a_{k+1}p+a_{k+2}p+\cdots +a_mp$ as required. Since $MT\big(\langle a_j\rangle_{j=1}^m,\langle y_n\rangle_{n=1}^\infty\big)\subseteq A\cap (0,\epsilon)\subseteq A$ and $MT\big(\langle a_j\rangle_{j=1}^m,\langle y_n\rangle_{n=1}^\infty\big)\in g(p)$, we have $A\in g(p)$.
\end{proof}
\begin{theorem}\label{4.9}
Let $S\subseteq (0,\infty)$ be a dense subsemigroup of both $([0,\infty),+)$ and $([0,\infty),\cdot)$. Let $m,l\in\mathbb{N}$. For each $i\in\{1,2,\ldots,l\}$, let $\langle x_{i,n}\rangle_{n=1}^\infty$ be a sequence in $S$ such that $\lim_{n\rightarrow\infty}x_{i,n}=0$ and let $p_i\in\bigcap_{r=1}^\infty cl_{\beta S}\big(FS\big(\langle x_{i,n}\rangle_{n=r}^\infty\big)\big)$. Let $f:\{1,2,\ldots,m\}\rightarrow\{1,2,\ldots,l\}$ be a function and let $g\in\mathbb{N}[x_1,x_2,\ldots,x_m]$ be a polynomial with $g\big(\overline{0}\big)=0$. Then for each $\epsilon>0$, $\big\{g\big(\sum_{n\in F_1}x_{f(1), n},$ $\sum_{n\in F_2}x_{f(2), n}, \ldots,\sum_{n\in F_m} x_{f(m), n}\big): \text{each } F_j\in\mathcal{P}_f(\mathbb{N}) \text{ and } F_1<F_2<\cdots<F_m\big\} \cap (0, \epsilon)\in g\big(p_{f(1)}, p_{f(2)}, \ldots, p_{f(m)}\big)$.
\end{theorem}
\begin{proof}
Let $\widetilde{g}:\beta\big(\bigtimes_{j=1}^m S\big)\rightarrow\beta S$ be the continuous extension of $g$. Then by \cite[Theorem 3.2]{bhw}, $\widetilde{g}\big(\otimes_{j=1}^mp_{f(j)}\big)=g\big(p_{f(1)},p_{f(2)},\ldots,p_{f(m)}\big)$. Let $A=\big\{\big(\sum_{n\in F_1}x_{f(1),n},\sum_{n\in F_2}x_{f(2),n},\ldots,\sum_{n\in F_m}x_{f(m),n}\big):\text{each } F_j\in\mathcal{P}_f(\mathbb{N}) \text{ and } F_1<F_2<\cdots<F_m\big\}$. 

By passing to a subsequence, we may presume that for each $i\in\{1,2,\ldots,l\}$, $\sum_{n=1}^\infty x_{i,n}$ converges, and then we have that $p_i\in 0_S^*$. By Lemma \ref{4.5}, $\otimes_{j=1}^m p_{f(j)}\in 0_{S^m}^*$. Since $g$ is continuous and $g\big(\overline{0}\big)=0$, so for each $\delta>0$, it follows easily that $[0,\epsilon)\cap S\in\widetilde{g}\big(\otimes_{j=1}^m p_{f(j)}\big)$, and thus $\widetilde{g}\big(\otimes_{j=1}^m p_{f(j)}\big)=g\big(p_{f(1)},p_{f(2)},\ldots, p_{f(m)}\big)\in 0_S^*$. Then, by Lemma \ref{4.7}, $A\in\otimes_{j=1}^m p_{f(j)}$, so by \cite[Lemma 3.30]{hs}, $g[A]\in\widetilde{g}\big(\otimes_{j=1}^m p_{f(j)}\big)$. Hence $g[A]\cap (0,\epsilon)\in g\big(p_{f(1)},p_{f(2)},\ldots,p_{f(m)}\big)$.
\end{proof}
\begin{theorem}\label{4.11}
Let $S\subseteq (0,\infty)$ be a dense subsemigroup of both $([0,\infty),+)$ and $([0,\infty),\cdot)$. Let $m,l\in\mathbb{N}$. For each $i\in\{1,2,\ldots,l\}$, let $\langle x_{i,n}\rangle_{n=1}^\infty$ be a sequence in $S$ such that $\lim_{n\rightarrow\infty}x_{i,n}=0$. Let $f:\{1,2,\ldots,m\}\rightarrow\{1,2,\ldots,l\}$ be a function, let $g\in\mathbb{N}[x_1,x_2,\ldots,x_m]$ be a polynomial with $g\big(\overline{0}\big)=0$ and let $A\subseteq S$. The following statements are equivalent.
\begin{enumerate}
    \item\label{4.11(1)} For each $i\in \{1, 2, \ldots, l\}$ there exists $p_i=p_i+p_i\in\bigcap_{r=1}^{\infty} cl_{\beta S}\big(FS(\langle x_{i, n}\rangle_{n=r}^{\infty}\big)\big)$ such that $A\in g(p_{f(1)}, p_{f(2)}, \ldots, p_{f(m)})$.
    \item\label{4.11(2)} For each $\epsilon >0$ and for each $i\in\{1, 2, \ldots, l\}$ there is a sum subsystem $FS\big(\langle y_{i, n}\rangle_{n=1}^\infty\big)$ of $FS\big(\langle x_{i, n}\rangle_{n=1}^\infty\big)$ such that  
    $\Big\{g\big(\sum_{n\in F_1}y_{f(1),n}, \sum_{n\in F_2}y_{f(2),n},$ $\ldots,\sum_{n\in F_m} y_{f(m),n}\big):\text{each } F_j\in\mathcal{P}_f(\mathbb{N}) \text{ and } F_1<F_2<\ldots<F_m\Big\} \subseteq A\cap (0, \epsilon)$. 
\end{enumerate}
\end{theorem}
\begin{proof}
(\ref{4.11(1)}) implies (\ref{4.11(2)}). By \cite[Theorem 3.2]{bhw} $g\big(p_{f(1)},p_{f(2)},\ldots,p_{f(m)}\big)=\widetilde{g}\big(\otimes_{j=1}^mp_{f(j)}\big)$ and from the proof of Theorem \ref{4.9}, we have that $\widetilde{g}\big(\otimes_{j=1}^mp_{f(j)}\big)\in 0_S^*$. Let $\epsilon>0$. Then $A\cap (0,\epsilon)\in\widetilde{g}\big(\otimes_{j=1}^m p_{f(j)}\big)$. Pick $B\in\otimes_{j=1}^m p_{f(j)}$ such that $\widetilde{g}\big[ cl_{\beta(S^m)}B\big]\subseteq cl_{\beta S}(A\cap (0,\epsilon))$. By Lemma \ref{4.6}, for each $i\in\{1,2,\ldots,l\}$ pick a sum subsystem $FS\big(\langle y_{i,n}\rangle_{n=1}^\infty\big)$ of $FS\big(\langle x_{i,n}\rangle_{n=1}^\infty\big)$ such that $\big\{\big(\sum_{n\in F_1}y_{f(1),n},\sum_{n\in F_2}y_{f(2),n},\ldots,$ $\sum_{n\in F_m}y_{f(m),n}\big):\text{each } F_j\in\mathcal{P}_f(\mathbb{N}) \text{ and } F_1<F_2<\cdots<F_m\big\}\subseteq B$. Then $\big\{g\big(\sum_{n\in F_1}y_{f(1),n},\sum_{n\in F_2}y_{f(2),n},\ldots,\sum_{n\in F_m}y_{f(m),n}\big):\text{each } F_j\in\mathcal{P}_f(\mathbb{N}) \text{ and }$ $F_1<F_2<\cdots<F_m\big\}\subseteq g[B]\subseteq A\cap (0,\epsilon)$.

(\ref{4.11(2)}) implies (\ref{4.11(1)}). Fix $\epsilon>0$. For each $i\in\{1,2,\ldots,l\}$ pick a sum subsystem $FS\big(\langle y_{i,n}\rangle_{n=1}^\infty\big)$ of $FS\big(\langle x_{i,n}\rangle_{n=1}^\infty\big)$ as guaranteed in (\ref{4.11(2)}), and pick by \cite[Lemma 5.11]{hs} $p_i=p_i+p_i\in\bigcap_{r=1}^\infty cl_{\beta S}\big(FS\big(\langle y_{i,n}\rangle_{n=r}^\infty\big)\big)\subseteq\bigcap_{r=1}^\infty cl_{\beta S}\big(FS\big(\langle x_{i,n}\rangle_{n=r}^\infty\big)\big)$. By Theorem \ref{4.9}, $\big\{g\big(\sum_{n\in F_1}y_{f(1),n},\sum_{n\in F_2}y_{f(2),n},\ldots,\sum_{n\in F_m}y_{f(m),n}\big): \text{each } F_j\in\mathcal{P}_f(\mathbb{N}) \text{ and } F_1<F_2<\cdots<F_m\big\}\cap(0,\epsilon)\in g\big(p_{f(1)},p_{f(2)},\ldots,p_{f(m)}\big)$, so $A\cap(0,\epsilon)\in g\big(p_{f(1)},p_{f(2)},\ldots,p_{f(m)}\big)$ and thus $A\in g\big(p_{f(1)},p_{f(2)},\ldots,p_{f(m)}\big)$.
\end{proof}
The following two theorems are generalizations of Theorem \ref{Theorem 2.8} for near zero semigroups.
\begin{theorem}\label{4.12}
Let $m\in\mathbb{N}$, and for $i\in\{1,2,\ldots,m\}$ let $S_i\subseteq (0,\infty)$ be a dense subsemigroup of $([0,\infty),+)$. Let $S=\bigtimes_{i=1}^m S_i$ and let $A\subseteq S$. The following statements are equivalent.
\begin{enumerate}
\item\label{4.12(1)} For each $i\in \{1,2,\ldots,m\}$, there is an idempotent $p_i\in 0_{S_i}^*$ such that $A\in\otimes_{i=1}^m p_i$.
\item\label{4.12(2)} For each $i\in\{1,2,\ldots,m\}$, there is a sequence $\langle x_{i,n}\rangle_{n=1}^\infty$ in $S_i$ such that $\sum_{n=1}^\infty x_{i,n}$ converges and $\big\{\big(\sum_{n\in F_1}x_{1,n},\sum_{n\in F_2}x_{2,n},\ldots,\sum_{n\in F_m}x_{m,n}\big): \text{each } F_i\in\mathcal{P}_f(\mathbb{N}) \text{ and } F_1<F_2<\cdots <F_m\big\}\subseteq A$.
\end{enumerate}
\end{theorem}
\begin{proof}
(\ref{4.12(1)}) implies (\ref{4.12(2)}). For $i\in\{1,2,\ldots,m\}$ pick an idempotent $p_i\in 0_{S_i}^*$ such that $A\in\otimes_{i=1}^m p_i$. By Lemma \ref{4.5}, $\otimes_{i=1}^m p_i\in \overline{0}_S^*$ so $\bigtimes_{i=1}^m \big(\big(0,\frac{1}{2}\big)\cap S_i\big)\in \otimes_{i=1}^m p_i$. Therefore $A\cap \bigtimes_{i=1}^m \big(0,\frac{1}{2}\big) \in \otimes_{i=1}^m p_i$. By \cite[Theorem 1.16]{bhw}, for each $i\in\{1,2,\ldots,m\}$ pick a sequence $\langle x_{i,n}\rangle_{n=1}^\infty$ in $S_i$ such that $\big\{\big(\sum_{n\in F_1}x_{1,n},\sum_{n\in F_2}x_{2,n},$ $\ldots,\sum_{n\in F_m}x_{m,n}\big): \text{each } F_i\in\mathcal{P}_f(\mathbb{N}) \text{ and } F_1<F_2<\cdots<F_m\big\}\subseteq A\cap \bigtimes_{i=1}^m\big(0,\frac{1}{2}\big)\subseteq A$. Then for each $i\in\{1,2,\ldots,m\}$ and each $F\in\mathcal{P}_f(\mathbb{N})$ with $\min F \geq i$, $\sum_{n\in F}x_{i,n}<\frac{1}{2}$ so $\sum_{n=1}^\infty x_{i,n}$ converges.

(\ref{4.12(2)}) implies (\ref{4.12(1)}). For each $i\in\{1,2,\ldots,m\}$, pick by \cite[Lemma 5.11]{hs} an idempotent $p_i\in\bigcap_{r=1}^\infty cl_{\beta S_i}\big(FS\big\langle x_{i,n}\rangle_{n=r}^\infty\big)\big)$. For each $i\in\{1,2,\ldots,m\}$, since $\sum_{n=1}^\infty x_{i,n}$ converges, so $p_i\in 0_{S_i}^*$. Let $l=m$ and let $f:\{1,2,\ldots,m\}\rightarrow\{1,2,\ldots,l\}$ be the identity function. Apply Lemma \ref{4.7}.
\end{proof}

\begin{theorem}\label{4.13}
Let $m\in\mathbb{N}$, let $S\subseteq (0,\infty)$ be a dense subsemigroup of $([0,\infty),+)$ and let $A\subseteq S^m$. The following statements are equivalent.
\begin{enumerate}
\item\label{4.13(1)} There is an idempotent $p\in 0_S^*$ such that $A\in \otimes_{i=1}^mp$.
\item\label{4.13(2)} There is a sequence $\langle x_n\rangle_{n=1}^\infty$ in $S$ such that $\sum_{n=1}^\infty x_n$ converges and $\big\{\big(\sum_{n\in F_1}x_n,\sum_{n\in F_2}x_n,\ldots,\sum_{n\in F_m}x_n\big): \text{each } F_i\in\mathcal{P}_f(\mathbb{N}) \text{ and } F_1<F_2<\cdots <F_m\big\}\subseteq A$.
\end{enumerate}
\end{theorem}
\begin{proof}
(\ref{4.13(1)}) implies (\ref{4.13(2)}). Pick an idempotent $p\in 0_S^*$ such that $A\in\otimes_{i=1}^mp$. By Lemma \ref{4.5}, $\otimes_{i=1}^mp\in \overline{0}_{S^m}^*$ so $\bigtimes_{i=1}^m\big(\big(0,\frac{1}{2}\big)\cap S\big)\in\otimes_{i=1}^mp$. Therefore $A\cap\bigtimes_{i=1}^m\big(0,\frac{1}{2}\big)\in\otimes_{i=1}^mp$. Pick by \cite[Theorem 1.17]{bhw} a sequence $\langle x_n\rangle_{n=1}^\infty$ in $S$ such that $\big\{\big(\sum_{n\in F_1}x_n,\sum_{n\in F_2}x_n,\ldots,\sum_{n\in F_m}x_n\big):\text{each } F_i\in\mathcal{P}_f(\mathbb{N}) \text{ and } F_1<F_2<\cdots<F_m\big\}\subseteq A\cap\bigtimes_{i=1}^m\big(0,\frac{1}{2}\big)\subseteq A$. Then for each $F\in\mathcal{P}_f(\mathbb{N})$ with $\min F\geq i$, $\sum_{n\in F}x_n<\frac{1}{2}$ so $\sum_{n=1}^\infty x_n$ converges.

(\ref{4.13(2)}) implies (\ref{4.13(1)}). Pick by \cite[Lemma 5.11]{hs} an idempotent $p\in\bigcap_{r=1}^\infty cl_{\beta S}\big(FS\big(\langle x_n\rangle_{n=r}^\infty\big)\big)$. Since $\sum_{n=1}^\infty x_n$ converges, so $p\in 0_S^*$. By Lemma \ref{4.7} with $l=1$, $A\in\otimes_{i=1}^m p$.
\end{proof}

\section*{Acknowledgments}
The authors are indebted to the anonymous reviewers for their generous comments and suggestions on the previous manuscripts. The first author gratefully acknowledges the grant UGC-NET SRF fellowship with Ref. No. 20/12/2015(ii)EU-V of CSIR-UGC NET December 2015.

\end{document}